\documentclass[11pt]{article}  
\usepackage{amsmath}
\usepackage{amssymb}
\usepackage{theorem}
\usepackage{euscript}
\usepackage{pstricks}
\usepackage{authblk}
\usepackage{verbatim}
\usepackage{hyperref}
\hypersetup
{
	colorlinks,
	citecolor=black,
	filecolor=black,
	linkcolor=black,
	urlcolor=blue
}
\hypersetup{linktocpage}
\newtheorem{theorem}{Theorem}[section]
\newtheorem{lemma}{Lemma}[section]
\newtheorem{corollary}{Corollary}[section]
\numberwithin{equation}{section}

\theoremstyle{definition}
 
\theoremstyle{remark}

\newcommand{\brac}[1]{\left(#1\right)}

\newcommand{\norm}[2]{\left\|{#1}\right\|_{#2}}

\newcommand{\abs}[1]{\left|#1\right|}

\newcommand{\ba}{{\boldsymbol{a}}}
\newcommand{\bb}{{\boldsymbol{b}}}

\newcommand{\be}{{\boldsymbol{e}}}

\newcommand{\bk}{{\boldsymbol{k}}}
\newcommand{\bell}{{\boldsymbol{\ell}}}

\newcommand{\bp}{{\boldsymbol{p}}}

\newcommand{\bs}{{\boldsymbol{s}}}

\newcommand{\bx}{{\boldsymbol{x}}}

\newcommand{\by}{{\boldsymbol{y}}}

\newcommand{\bz}{{\boldsymbol{z}}}

\newcommand{\bW}{{\boldsymbol{W}}}

\newcommand{\brho}{{\boldsymbol{\rho}}}
\newcommand{\bsigma}{{\boldsymbol{\sigma}}}
\newcommand{\bphi}{{\boldsymbol{\phi}}}

\newcommand{\rd}{{\rm d}} 
 
\newcommand{\support}{{\rm support}} 

\def\ZZ{{\mathbb Z}}

\def\RR{{\mathbb R}}
\def\RRd{{\mathbb R}^d}

\def\NN{{\mathbb N}}
\def\NNd{{\NN}^d}

\def\NN{{\mathbb N}}
\def\RR{{\mathbb R}}

\def\FF{{\mathcal F}}

\def\NNd{{\mathbb N}^d}
\def\RRd{{\mathbb R}^d}
\def\RRm{{\mathbb R}^m}

\def\IIi{{\mathbb I}^\infty}

\def\NMO{{\mathbb N}_0^M}
\def\RRi{{\mathbb R}^\infty}

\def\RRM{{\mathbb R}^M}
\def\RRm{{\mathbb R}^m}

\def\Ee{{\mathcal E}}
\def\Ff{{\mathcal F}}

\def\Jj{{\mathcal J}}

\def\Ll{{\mathcal L}}
\def\Nn{{\mathcal N}}
\def\Oo{{\mathcal O}}

\def\ZZ{{\mathbb Z}}
\def\NN{{\mathbb N}}
\def\RR{{\mathbb R}}

\def\FF{{\mathbb F}}

\def\NNd{{\mathbb N}^d}
\def\RRd{{\mathbb R}^d}

\def\RRi{{\mathbb R}^\infty}

\def\supp{\operatorname{supp}}
\def\dv{\operatorname{div}}

\title{\sffamily Collocation approximation by deep neural ReLU networks for 
 parametric  and stochastic  PDEs with lognormal inputs } 

\author{Dinh D\~ung}
\affil{Information Technology Institute, Vietnam National University, Hanoi
	\protect\\
	144 Xuan Thuy, Cau Giay, Hanoi, Vietnam
	\protect\\
	Email: dinhzung@gmail.com}

\date{\today}
\begin{document}
\maketitle

\begin{abstract}
	We  obtained convergence rates of the collocation approximation by deep ReLU neural networks of  solutions to   elliptic PDEs with lognormal inputs, parametrized by $\by$ from the non-compact  set $\RRi$. The approximation error is measured in the norm of the Bochner space $L_2(\RRi, V, \gamma)$, where $\gamma$ is the infinite tensor product  standard Gaussian probability measure on $\RRi$ and $V$ is the energy space.
	We also obtained similar dimension-independent results  for the case when the  lognormal inputs are parametrized on $\RRM$ with  very large dimension $M$, and the approximation error  is measured in the $\sqrt{g_M}$-weighted  uniform norm of the Bochner space $ L_\infty^{\sqrt{g}}(\RRM, V)$, where $g_M$ is the  density function of the standard Gaussian probability measure on $\RRM$.

	
	
\end{abstract}

\section{Introduction}

Partial differential equations (PDEs) with parametric and stochastic inputs are a common model used  in science
and engineering. Stochastic nature reflects the uncertainty in various parameters
presented in the physical phenomenon modelled  by the equation.
A central problem of computational uncertainty quantification is  efficient   numerical approximation for  parametric and stochastic PDEs  which has been of great interest and achieved  significant progress in recent decades. There is a  large number of non-deep-neural-network papers on this topic to mention all of them. We  point out just some works \cite{BCDC17,BCDM17,BCM17,CCDS13,CCS13,CCS15,CoDe15a,CDS10,CDS11,Dung19,Dung21, EST18,HoSc14,ZDS19,ZS20} which are directly related to our paper.  In particular,  collocation approximations which are based on a finite number of particular solvers to parametric and stochastic PDEs, were considered in \cite{CCS13,CCS15,CoDe15a,Dung19,Dung21,DNSZ,EST18,ZDS19}.

The approximation universality of neural networks has been achieved a basis understanding since the 1980's (\cite{Bar91,Cyben89,Funa90,HSW89}).
Deep neural networks in recent years have been rapidly developed in theory and applications  to a wide range of fields due to their advantage  over shallow ones. 
Since their application range is getting wider, theoretical analysis discovering reasons of these significant
practical improvements attracts special attention  \cite{ ABMM17,DDF.19,MPCB14,Te15,Te16}. 
In recent years, there has been a number of interesting papers that addressed the role of depth and architecture of deep neural networks for non-adaptive and adaptive approximation of  functions   having a particular regularity 
\cite{AlNo20,EWa18, GKNV19,GKP20,GPEB19,Mha96,LSYZ21, SYZ20,PeVo18,Ya17a,Ya18}. High-dimensional approximations by deep neural networks have been studied in \cite{MoDu19, Suzu18,DN20,DKT21}, and  their applications to high-dimensional PDEs in  \cite{EGJS18,GS20,GS21,GH21,HOS21,OSZ21,SS18}.  Most of  these papers employed  the rectified linear unit (ReLU) as the activation function of deep neural networks  since the ReLU is a simple and preferable  in many applications. The output of such a deep neural network is a continuous piece-wise linear function which is easily and cheaply  computed. The reader can consult the recent survey papers  \cite{DHP21,Pet20} for various problems and aspects of neural network approximation and bibliography.

Recently,  a number of papers have been devoted to various problems and methods of  deep neural network approximation for parametric and stochastic PDEs such as dimensionality reduction \cite{TB18}, deep neural network expression rates for generalized polynomial chaos expansions  (gpc) of solutions to parametric elliptic PDEs \cite{DNP21,SZ19}, reduced basis methods \cite{KPRS21} the problem of learning the discretized parameter-to-solution map in practice \cite{GPR.20}, Bayesian PDE inversion \cite{HOS21,HSZ20,OSZ20}, etc. Note that except \cite{DNP21} all of these papers treated parametric and stochastic PDEs with affine inputs on the compact set $\IIi:= [-1,1]^\infty$.
 The authors of paper \cite{SZ19}  proved dimension-independent deep neural network expression rate bounds of the uniform approximation of solution to parametric elliptic PDE with affine inputs on $\IIi$ based on   $n$-term truncations of the  non-orthogonal Taylor  gpc expansion. The construction of approximating deep neural networks relies on weighted summability of the Taylor gpc expansion coefficients of the solution which is derived from its analyticity. The paper \cite{DNP21} investigated  non-adaptive methods of deep ReLU neural network approximation of the solution $u$ to  parametric and stochastic elliptic PDEs with lognormal inputs on non-compact set $\RRi$. The approximation error is measured in the norm of the Bochner space $L_2(\RRi, V, \gamma)$, where $\gamma$ is the tensor product   standard Gaussian probability on $\RRi$ and $V$ is the energy space. The approximation is based  on an $m$-term truncation of the  Hermite gpc of $u$. Under a certain assumption on $\ell_q$-summability ($0<q<\infty$) for the lognormal inputs,  it was proven that for every integer $n > 1$, one can  construct  a non-adaptive compactly supported deep ReLU neural network $\bphi_n$ of size $\le n$ on $\RRm$ with $m = \Oo(n/\log n)$, having $m$ outputs so that the summation  constituted by replacing Hermite polynomials in the $m$-term truncation by these $m$ outputs approximates $u$ with the error bound 
$\Oo\brac{\left(n/\log n\right)^{-1/q}}$.	
The authors of \cite{DNP21} also obtained some results on similar problems for parametric and stochastic elliptic PDEs with affine inputs, based on the Jacobi and Taylor gpc expansions.

In the present paper, we are interested in    constructing  deep ReLU neural networks for  collocation approximation of 
 the solution to parametric  elliptic PDEs with lognormal inputs.  We study the convergence rate of this approximation in terms of the size of  deep ReLU neural networks.

Let $D \subset \RR^d$ be a bounded  Lipschitz domain.  Consider the diffusion elliptic equation 
\begin{equation} \label{ellip}
- \dv (a\nabla u)
\ = \
f \quad \text{in} \quad D,
\quad u|_{\partial D} \ = \ 0, 
\end{equation}
for  a given  right-hand side $f$ and 
 diffusion coefficient $a$ as  functions on $D$.
Denote by $V:= H^1_0(D)$ the so-called energy space  of all those functions from the Sobolev space $H^1(D)$ that have compact support in $D$.  Let $H^{-1}(D)$ be the dual space of $V$. Assume that  $f \in H^{-1}(D)$ (in what follows this preliminary assumption always holds without mention). If $a \in L_\infty(D)$ satisfies the ellipticity assumption
\begin{equation} \nonumber
0<a_{\min} \leq a \leq a_{\max}<\infty,
\end{equation}
by the well-known Lax--Milgram lemma, there exists a unique 
solution $u \in V$  to the (non-parametric) equation~\eqref{ellip}  in the weak form
\begin{equation} \nonumber
\int_{D} a\nabla u \cdot \nabla v \, \rd \bx
\ = \
\langle f , v \rangle,  \quad \forall v \in V.
\end{equation}

Partial differential equations with parametric and stochastic inputs are a common model used  in science
and engineering.  
For the equation~\eqref{ellip},
we consider diffusion coefficients having a parametrized form $a=a(\by)$, where $\by=(y_j)_{j \in \NN}$
is a sequence of real-valued parameters ranging in the set 
$\RRi$.
Denote by $u(\by)$ the solution to the 
parametrized  diffusion elliptic equation 
\begin{equation} \label{p-ellip}
- {\rm div} (a(\by)\nabla u(\by))
\ = \
f \quad \text{in} \quad D,
\quad u(\by)|_{\partial D} \ = \ 0. 
\end{equation}	
The resulting solution operator maps
$\by\in \RRi \mapsto u(\by)\in V$. The goal is to 
achieve numerical 
approximation of this complex map by a small number of parameters with a
guaranteed error in a given norm. Depending on the 
nature of the modeled object, the parameter $\by$ may be 
either deterministic or random.
In the present paper, we consider the  so-called lognormal case when the diffusion coefficient $a$  is of the form
\begin{equation} \label{lognormal}
a(\by)=\exp(b(\by))
\end{equation}
with $ b(\by)$ in the infinite-dimensional form:
\begin{equation} \label{b(y)infinite}
\ b(\by)=\sum_{j = 1}^\infty y_j\psi_j, \quad \by \in \RRi,
\end{equation}
where the $y_j$ are i.i.d. standard Gaussian random 
variables  and $\psi_j \in L_\infty(D)$. We also consider the finite-dimensional form when
\begin{equation} \label{b(y)M}
\ b(\by)=\sum_{j = 1}^M y_j\psi_j, \quad \by \in \RRM,
\end{equation}
with finite but very large dimension $M$. Notice that for a fixed $\by$ both the cases \eqref{b(y)infinite} and \eqref{b(y)M} of  equation \eqref{p-ellip} satisfy the ellipticity assumption, and therefore there exists  exists a unique 
solution $u(\by) \in V$  to the equation~\eqref{p-ellip}  in the weak form. However, there is no the  uniform ellipticity with respect to $\by$ since $\RRi$ and $\RRM$ are not compact sets.

We briefly describe the main results of  the present paper.  

We investigate  non-adaptive collocation methods of high-dimensional deep ReLU neural network approximation of the solution $u(\by)$ to  parametrized diffusion elliptic PDEs \eqref{p-ellip} with lognormal inputs \eqref{lognormal} in the infinite-dimensional case  \eqref{b(y)infinite} and finite-dimensional case \eqref{b(y)M}. 
In the infinite-dimensional case  \eqref{b(y)infinite}, the approximation error is measured in the norm of the Bochner space $ L_2(\RRi, V, \gamma)$, where $\gamma$ is the infinite tensor product   standard Gaussian probability on $\RRi$.
Assume that there exists an increasing sequence of positive numbers strictly larger than one
$\brho = \brac{\rho_{j}}_{j \in \NN}$ such that for some $0<q<2$,
$$
\norm{\sum _{j \in \NN} \rho_j |\psi_j|}{L_\infty(D)} <\infty \ \ {\rm and} \ \ \brho^{-1}= \brac{\rho_{j}^{-1}} _{j \in \NN}\in {\ell_q}(\NN).
$$ 
Then, given
an arbitrary number $\delta$ with $0 < \delta < \min \brac{1, 1/q -1/2}$,
for every integer $n > 1$, we can  construct  a  deep ReLU neural network 
$\bphi_n:= \big(\phi_j\big)_{j=1}^m$ on $\RRm$ with $m= \Oo \brac{n^{1 - \delta}}$ of size at most $n$ and a sequence of points $Y_n:=\brac{\by^j}_{j=1}^m \subset \RRm$ so that

\begin{itemize}
	\item [{\rm (i)}]
	The deep ReLU neural network $\bphi_n$ and  sequence of points $Y_n$ are independent of $u$;
	\item [{\rm (ii)}] 
	The  output dimension of $\bphi_n$ is $m = \Oo \brac{n^{1 - \delta}}$;	
	\item [{\rm (iii)}] 
	The  depth of $\bphi_n$ is   $\Oo(n^\delta)$;
	\item [{\rm (iv)}] 
	The  components $\phi_j$, $j = 1,...,m$, of $\bphi_n$ are deep ReLU neural networks on 
	$\RR^{m_j}$ with $m_j  = \Oo(n^\delta)$, having support of contained in the super-cube $[-T,T]^{m_j}$ with 
	$T =  \Oo \brac{n^{1 - \delta}}$;
	\item [{\rm (v)}]
	If $\Phi_j$ is the extension of $\phi_j$ to the whole $\RRi$ by $\Phi_j(\by) =\phi_j\brac{\big(y_j\big)_{j=1}^{m_j}}$ for 
	$\by = \big(y_j\big)_{j\in \NN} \in \RRi$,
	 the collocation approximation of $u$ by the function
	$$
\Phi_n u:=	\sum_{j=1}^m u\brac{\by^j} \Phi_j,
	$$
	which is  based on the $m$ solvers $\brac{u\brac{\by^j}}_{j=1}^m$ and the deep ReLU network $\bphi_n$,
	gives the twofold error estimates
	\begin{align} \label{errors}
	\norm{u- \Phi_n u}{L_2(\RRi, V, \gamma)} 
	= \Oo\brac{m^{- \brac{\frac{1}{q} - \frac{1}{2}}}} 
	=\Oo\brac{n^{- (1-\delta)\brac{\frac{1}{q} - \frac{1}{2}}}}.
	\end{align}
\end{itemize}

We also obtained similar results in manner of the items (i)--(v) in the finite-dimensional case  \eqref{b(y)M} with  the approximation error  measured in the $\sqrt{g_M}$-weighted uniform norm of the Bochner space $ L_\infty^{\sqrt{g}}(\RRM, V)$, where $g_M$ is the  density function of the standard Gaussian probability measure on $\RRM$.

These results are derived from results on deep ReLU neural network collocation approximation of functions in Bochner spaces 
related to a general separable Hilbert space and standard Gaussian probability measures based on weighted $\ell_2$-summabilities of the Hermite gpc expansion coefficients of functions (see Section \ref{Deep ReLU neural network approximation in Bochner spaces} for details).  

Notice that the error bound in $m$ in \eqref{errors}  is the same as the error bound  of  the   collocation approximation of $u$ by the sparse-grid Lagrange gpc interpolation  based on $m$ the same particular solvers  $\brac{u\brac{\by^j}}_{j=1}^m$, which so far is the best known result \cite[Corollary 3.1]{Dung21}. Moreover, the convergence rate $(1-\delta)(1/q - 1/2)$ with arbitrarily small $\delta > 0$ in terms of the size  of the deep ReLU network in the collocation approximation,  is comparable with the convergence rate $1/q - 1/2$  with respect to the number  of particular solvers  in the collocation approximation by  sparse-grid Lagrange gpc interpolation. This is a crucial difference between the results of the present paper and of \cite{DNP21} which proved the convergence rate of the deep ReLU network approximation of solutions to parametrized diffusion elliptic PDEs \eqref{p-ellip} with lognormal inputs \eqref{lognormal} based on a different input information --  the coefficients of  Hermite gpc  expansion  in its finite truncations. Although that convergence rate is sharper than one in \eqref{errors}, in general, it is well-known that collocation approximations are more important, difficult and applicable than those  using  spectral information about the coefficients of an orthonormal expansion. The extension of  the results  (i)--(v) to the Bochner space $ L_\infty^{\sqrt{g}}(\RRM, V)$ is also an important difference of our contribution comparing with \cite{DNP21}.

We would like to emphasize that the motivation of this paper is to establish approximation results which should  show posibilities of non-adaptive collocation approximation by deep ReLU neural networks and  convergence rates of approximation for the parametrized diffusion elliptic equation \eqref{p-ellip} with lognormal inputs, and we do not consider the numerical aspect of the problem. 
The results  themselves do not give a practically realizable approximation because they do not cover the approximation of the coefficients which are particular solvers at certain points of the spatial variables.  Moreover, the approximant $\Phi_n u$ is not a real deep ReLU networks, but just a combination of these particular solvers and the components of a deep ReLU network. It would be interesting to investigate the problem of full deep ReLU neural network approximation of the solution $u$ to  parametric and stochastic elliptic PDEs by combining the spatial and parametric domains based on fully discrete approximation  in \cite{BCDC17,Dung21}. This problem will be discussed in a forthcoming paper. 

The paper is organized as follows. 
In Section \ref{Deep ReLU neural networks}, we present a necessary knowledge  about deep ReLU neural networks.
Section \ref{Deep ReLU neural network approximation in Bochner spaces}
is devoted to collocation methods of deep ReLU neural network approximation of functions in Bochner spaces 
$L_2(\RRi, X, \gamma)$ or in $L_2( \RRM,X,\gamma)$ related to a separable Hilbert space $X$  and the tensor product standard Gaussian probability measure $\gamma$.  
In Section \ref{lognormal inputs}, we apply the results in the previous section to the collocation approximation by deep ReLU neural networks  of the solution $u$ to the parametrized elliptic PDEs \eqref{p-ellip} with lognormal inputs \eqref{lognormal} on  in the infinite case \eqref{b(y)infinite} and  finite case \eqref{b(y)M}.

\medskip
\noindent
{\bf Notation} \  As usual, $\NN$ denotes the natural numbers, $\ZZ$  the integers, $\RR$ the real numbers and $ \NN_0:= \{s \in \ZZ: s \ge 0 \}$.
We denote  $\RR^\infty$ the
set of all sequences $\by = (y_j)_{j\in \NN}$ with $y_j\in \RR$.  For a set $G$, we denote by
$|G|$ the cardinality of $G$. 
If  $\ba= (a_j)_{j \in \Jj}$  is a sequence of positive numbers with  any index set $\Jj$, then we use the notation 
$\ba^{-1}:= (a_j^{-1})_{j \in \Jj}$.
We use letters $C$  and $K$ to denote general 
positive constants which may take different values, and 
$C_{\alpha,\beta,...}$ and $K_{\alpha,\beta,...}$  when we 
want to emphasize the dependence of these  constants on 
$\alpha,\beta,...$, or when this dependence is 
important in a particular situation.

For convenience to the reader, we list some specific notations and definitions  which widely used in the present paper and indicate where they are introduced. 

Section \ref{Deep ReLU neural networks}: The symbols $W(\Phi)$, $L(\Phi)$ and $\supp(\Phi)$ denote the size, the depth and the support of the deep ReLU neural network $\Phi$, respectively;  
$\sigma(t):= \max\{t,0\}$ is ReLU activation function. 

Section \ref{Tensor product Gaussian measures and Bochner spaces}: Denote by $\FF$  the set of all sequences of non-negative integers $\bs=(s_j)_{j \in \NN}$ such that their support $\supp (\bs):= \{j \in \NN: s_j >0\}$ is a finite set. Letter $J$ denotes either $\infty$ or $M \in \NN$; the set $U$ is defined in \eqref{set:U}, the set $\Ff$ in \eqref{set:Ff} and the set $\Nn$ in \eqref{set:Nn}:  $\gamma$ and $\gamma_M$ are the standard Gaussian meesures in $\RRi$ and $\RRM$, respectively.  For $\bs \in \FF$, put $|\bs|_1:= \sum_{j \in \NN} s_j$ and $|\bs|_0:= |\supp (\bs)|$. For $\bs, \bs' \in \Ff$,  the inequality $\bs' \le \bs$ means that $s_j' \le s_j$, $j \in \Nn$.   A set
$\bsigma=(\sigma_\bs)_{\bs \in \Ff}$ with $\sigma_\bs \in \RR$ is called increasing if 
$\sigma_{\bs'} \le \sigma_\bs$ for $\bs' \le \bs$.  The Bochner space $\Ll(U,X)$ is defined in \eqref{space:Ll(U,X)};  the Bochner spaces  $L_2(U,X,\gamma)$ and $L_{\infty}^{\sqrt{g}}(\RRM,X)$ are given by \eqref{L_2(U,X,gamma)} and \eqref{L_{infty}^{sqrt{g}}}, respectively; In \eqref{hermite},  $H_\bs$ is defined as the $\bs$th Hermite  orthnormal polynomial and   $v_\bs$ as the $\bs$th coefficient of  the Hermite  gpc expansion of $v$. 

Section \ref{ Hermite gpc interpolation approximation}: $Y_m = (y_{m;k})_{k \in \pi_m}$ is the increasing sequence of  the $m+1$ roots of the Hermite polynomial $H_{m+1}$;   $I_m$ is the  Lagrange intepolation operator defined by \eqref{I_(v)}; 
$\lambda_m$ is the Lebesgue constant defined by \eqref{lambda_m}; $\Delta_\bs$ is the tensor product operator defined by \eqref{Delta_bs(v)}; $I_\Lambda$  is the gpc interpolation operator defined by \eqref{I_Lambda}; the set $\bp(\theta, \lambda):= \brac{p_\bs(\theta, \lambda) }_{\bs \in \Ff}$ is defined  by \eqref{[p_s]}; the set $\Lambda(\xi)$ is defined by  \eqref{Lambda(xi)} and the set $G(\xi)$ by \eqref{G(xi):=}.

\section{Deep ReLU neural networks}
\label{Deep ReLU neural networks}

In this section, we present some auxiliary knowledge on deep ReLU neural networks which will be used as a tool of approximation. As in \cite{Ya17a}, we will use such deep feed-forward neural networks that
allows  connections between neurons  in a layer with neurons in any preceding  layers (but not in the same layer). 
 The ReLU activation function  is defined by 
$\sigma(t):= \max\{t,0\}, t\in \RR$.  
We denote:
$\sigma(\bx):= (\sigma(x_1),\ldots, \sigma(x_d))$ for $\bx=(x_1,\ldots,x_d) \in \RRd$.

Let us recall a standard definition of  deep ReLU  neural network and relevant terminology.  	Let $d,L\in \NN$, $L\geq 2$, $N_0=d$,  and $N_1,\ldots,N_{L}\in \NN$. Let 
	$\bW^\ell=\brac{w^\ell_{i,j}}\in \RR^{N_\ell\times \brac{\sum_{i=1}^{\ell-1}N_i}}$, $\ell=1,\ldots,L$, be 
	an $N_\ell\times \brac{\sum_{i=1}^{\ell-1}N_i}$ matrix, and $\bb^\ell =(b^\ell_j)\in \RR^{N_\ell}$.    A ReLU  neural network  $\Phi$ (on $\RRd$) with input dimension $d$, output dimension  $N_L$ and $L$ layers
	is  called a sequence of matrix-vector tuples
	$$
	\Phi=\brac{(\bW^1,\bb^1),\ldots,(\bW^L,\bb^L)},
	$$
	in which the following
	computation scheme is implemented:
	\begin{align*}
	\bz^0&: = \bx \in \RR^d;
	\\
	\bz^\ell &: = \sigma\brac{\bW^{\ell}\brac{\bz^0,\ldots,\bz^{\ell-1}}^{{\rm T}}+\bb^\ell}, \ \ \ell=1,\ldots,L-1;
	\\
	\bz^L&:= \bW^L{\brac{\bz^0,\ldots,\bz^{L-1}}^{{\rm T}}} + \bb^L.
	\end{align*}
	We call $\bz^0$ the input  and with an 
	ambiguity we use the notation   
	$\Phi(\bx):= \bz^L$ for the output of $\Phi$ which is an $L$-dimensional vector-function on $\RRd$. In some places we identify a ReLU neural network  with its output. 
	We adopt the following terminology.
	\begin{itemize}
		\item The number of layers $L(\Phi)=L$  is the depth of $\Phi$;
		\item The number of nonzero $w^\ell_{i,j}$ and $b^\ell_j$  is the  size of $\Phi$ and denoted by $W(\Phi)$;
		\item When $L(\Phi) \ge 3$, $\Phi$ is called a deep ReLU neural network, and otherwise, a shallow ReLU neural network.
		\item If $\Phi(\bx)=(\phi_j(\bx))_{j=1}^L$, the support of the deep ReLU neural network $\Phi$ is defined as $\bigcup_{j=1}^L \supp(\phi_j)$ and denoted by $\supp(\Phi)$.
	\end{itemize}

There are two basic operations which  neural networks allow for. This is the parallelelization of several  neural networks and the concatenation of two neural networks. The reader can find for instance, in  \cite{GKP20} (see also \cite{DHP21,Pet20}) for detailed decriptions as well as the following two lemmata  on these operations.

\begin{lemma}[Parallelization]\label{lem:parallel}
	Let $N\in \NN$,  $\lambda_j\in \RR$, $j=1,\ldots,N$. Let $\Phi_j$, $j=1,\ldots,N$ be deep ReLU neural networks with input dimension $d$. Then  we can explicitly construct a deep ReLU neural network  denoted by $\Phi$ so that 
	$$
	\Phi(\bx)
	=
	\sum_{j=1}^N\lambda_j \Phi_j(\bx),\quad \bx\in \RR^d.
	$$ 
	Moreover, we have 
	$$
	W(\Phi)\ \le \ \sum_{j=1}^N W(\Phi_j)
	\qquad \text{and}\qquad
	L(\Phi) \ = \ \max_{j=1,\ldots,N} L(\Phi_j). 
	$$
	The deep ReLU neural network $\Phi$ is called the parallelization of  $\Phi_j$, $j=1,\ldots,N$.
\end{lemma}

\begin{lemma}[Concatenation]\label{lem-concatenation} 
	Let $\Phi_1$ and $\Phi_2$ be two ReLU neural networks   such that output layer of $\Phi_1$ has the same dimension as input layer of $\Phi_2$. Then, we can explicitly construct a ReLU neural network $\Phi$ such that $\Phi(\bx)=\Phi_2(\Phi_1(\bx))$ for $\bx\in \RR^d$. Moreover we have
	$$W(\Phi)\leq 2W(\Phi_1)+2W(\Phi_2) \qquad \text{and}\qquad L(\Phi) = L(\Phi_1)+L(\Phi_2).$$
	The deep ReLU neural network $\Phi$ is called the concatenation of  $\Phi_1$ and $\Phi_2$.
\end{lemma}

The following lemma is a direct consequence of  \cite[Proposition 3.3]{SZ19}.

\begin{lemma} \label{lem:multi}
Let  $\bell \in \NNd$. For  every $\delta \in (0,1)$, we can explicitly construct  a deep  ReLU neural network  $\Phi_P$  on $\RRd$ so that 
	$$
	\sup_{ \bx \in [-1,1]^d} \Bigg|\prod_{j=1}^d x_j^{\ell_j} - \Phi_P(\bx) \Bigg| \leq \delta. 
	$$
	Furthermore, if $x_j=0$ for some $j\in \{1,\ldots,d\}$ then $\Phi_P(\bx)=0$ and there exists a constant $C>0$ independent of $\delta$, $d$ and $\bell$ such that
	$$
	W(\Phi_P) \leq C |\bell|_1\log (|\bell|_1\delta^{-1}) 
	\qquad \text{and}\qquad
	L(\Phi_P)  \leq C\log |\bell|_1\log(|\bell|_1\delta^{-1}) \,.
	$$
\end{lemma}


	For $j=0,1$, let $\varphi_j$ be the continuous piece-wise linear functions with break points $\{-2, -1,1,2\}$ and $\supp(\varphi_j) \subset [-2,2]$ such that $\varphi_0(x)=1$ and $\varphi_1(x)=x$ if $x\in [-1,1]$.	

\begin{lemma} \label{lem-product-varphi}
Let $\bell \in \NNd$ and $\varphi$ be either $\varphi_0$ or $\varphi_1$.
	For  every $\delta \in (0,1)$, we can explicitly construct  a deep  ReLU neural network  $\Phi$  on $\RRd$ so that 
	$$
	\sup_{ \bx \in [-2,2]^d} \Bigg|\prod_{j=1}^d\varphi^{\ell_j}(x_j) - \Phi(\bx) \Bigg| \leq \delta. 
	$$
	Furthermore, $\supp(\Phi )\subset [-2,2]^d$ and there exists a constant $C>0$ independent of $\delta$, $d$ and $\bell$ such that
	\begin{equation}\label{eq-varphi}
	W(\Phi) \leq C\big(1+ |\bell|_1\log (|\bell|_1\delta^{-1}) \big)
	\qquad \text{and}\qquad
	L(\Phi)  \leq C\big(1+\log |\bell|_1\log(|\bell|_1\delta^{-1})) \,.
	\end{equation}
\end{lemma}
\begin{proof}
Notice that  the explicit forms  of $\varphi_j$ via the ReLU activation function  are
	$$
	\varphi_0(x)= \sigma(x-2) - 3\sigma(x-1) + 4\sigma(x) - 3\sigma(x+1) + \sigma(x+2), 
	$$
	and 
	$$
	\varphi_1(x)= \sigma(x-2) - 2\sigma(x-1) +  2\sigma(x+1) -\sigma(x+2).
	$$
	This yields that $\varphi_j $ can be realized exactly by a shallow  ReLU neural network (still denoted by  $\varphi_j$) with size $W(\varphi_0)\leq 10$ and $W(\varphi_1)\leq 8$. 
	The network $\Phi$ can be constructed as a concatenation of deep ReLU neural networks $\{\varphi(x_j)\}_{j=1}^d$ and $\Phi_P$. By the definitions of deep ReLU neural network and function $\varphi$  we have  $$
	\bz^1 = \{\varphi(x_j)\}_{j=1}^d \subset [-1,1]^d.
	$$
	 Hence, the estimates \eqref{eq-varphi} follow directly from Lemmata  \ref{lem-concatenation} and \ref{lem:multi}. 
	\hfill
	\end{proof}

\section{Deep ReLU neural network approximation in Bochner spaces}
\label{Deep ReLU neural network approximation in Bochner spaces}

In this section, we investigate  collocation methods of deep ReLU neural network approximation of functions in Bochner spaces 
 related to a Hilbert space $X$  and tensor product standard Gaussian probability measures $\gamma$.  Functions to be approximated have the weighted $\ell_2$-summable Hermite gpc expansion coefficients (see Assumption (I) below).
The approximation is based  on the  sparse-grid Lagrange gpc  interpolation. 
We  construct  such methods and prove convergence rates of the approximation by them. 
The results obtained in this section will be applied to deep ReLU neural network collocation approximation of the solution of parametrized elliptic PDEs with lognormal inputs in the next section.

\subsection{Tensor product Gaussian measures and Bochner spaces}
\label{Tensor product Gaussian measures and Bochner spaces}

Let $\gamma(y)$ be the standard Gaussian probability measure on $\RR$ 
with the  density 
\begin{equation} \label{g}
 g(y):=\frac 1 {\sqrt{2\pi}} e^{-y^2/2}, \ {\rm i.e.,} \quad  \rd\gamma(y):=g(y)\,\rd y .
\end{equation}
For $M \in \NN$, the standard Gaussian probability measures $\gamma(\by)$ on $\RRM$ can be defined by
\begin{equation} \nonumber
\rd \gamma(\by) 
:= \ 
g_M(\by) \rd (\by)
\ = \
\bigotimes_{j=1}^M g(y_j) \rd (y_j), \quad
\by = (y_j)_{j=1}^M \in \RRM,
\end{equation}
where $g_M(\by) := \ \bigotimes_{j=1}^M g(y_j)$.

We next  recall a concept of
standard Gaussian probability measure $\gamma(\by)$ on $\RRi$ as 
the infinite tensor product of the standard Gaussian probability 
measures $\gamma(y_i)$:
\begin{equation} \nonumber
\gamma(\by) 
:= \ 
\bigotimes_{j \in \NN} \gamma(y_j) , \quad \by = (y_j)_{j \in \NN} \in \RRi.
\end{equation}
The sigma algebra for $\gamma(\by)$ is generated by the set of cylinders $A:= \prod_{j \in \NN} A_j$, where $A_j \subset \RR$ are univariate $\gamma$-measurable sets and only a finite number of $A_i$ are different from $\RR$. For such a set $A$, we have $\gamma(A) = \prod_{j \in \NN} \gamma(A_j)$.
(For details on infinite tensor product of probability measures, see, e.g., \cite[pp. 429--435]{HS65}.)

In what follows,  we use   the joint notation:  $J$ denotes either $\infty$ or $M \in \NN$ and
\begin{equation} \label{set:U}
	U
	:=
	\begin{cases}
		\RRM\ \ &{\rm if} \ \ J=M, \\
		\RRi \ \ &{\rm if} \ \ J= \infty;
	\end{cases}
\end{equation}
If $X$ is a separable Hilbert space, 
the standard Gaussian probability measure $\gamma$ on $U$ induces  the Bochner space $L_2(U,X,\gamma)$ of  $\gamma$-measurable mappings $v$ from $U$ to $X$, equipped with the norm
\begin{equation} \label{L_2(U,X,gamma)}
\|v\|_{L_2(U,X,\gamma)}
:= \
\left(\int_{U} \|v(\cdot,\by)\|_X^2 \, \rd \gamma(\by) \right)^{1/2}.
\end{equation}
For a $\gamma$-measurable subset $\Omega$ in $U$  the spaces  $L_2(\Omega,X,\gamma)$ and $L_2(\Omega,\gamma)$ is defined in the usual way. 

In the case $U=\RRM$, we introduce  also the space $L_{\infty}^{\sqrt{g}}(\RRM,X)$ as the set of all strongly $\gamma$-measurable functions  $v: \RRM \to X$ for which the $\sqrt{g_M}$-weighted uniform norm
\begin{equation} \label{L_{infty}^{sqrt{g}}}
\displaystyle
\|v\|_{L_{\infty}^{\sqrt{g}}(\RRM,X)}
:= \
\displaystyle
\operatornamewithlimits{ess \, sup}_{\by \in \RRM}  \left(\|v(\by)\|_X \sqrt{g_M(\by)}\right) \ < \ \infty.
\end{equation}
One may expect an infinite-dimensional version of  this space. Unfortunately, we could not give a correct definition of  space $L_{\infty}^{\sqrt{g}}( \RRi,X)$ because there is no an infinite-dimensional counterpart of the weight $g_M$. However, under certain assumptions (see Assumption (I) in Subsection \ref{ Hermite gpc interpolation approximation}), we  can obtain some approximation results which do not depend on $M$, in particular, when $M$ are very large.
 We make use of the abbreviations: 
$L_{\infty}^{\sqrt{g}}(\RRM)= L_{\infty}^{\sqrt{g}}(\RRM,\RR)$ and 
$L_{\infty}^{\sqrt{g}}(\RR)= L_{\infty}^{\sqrt{g}}(\RR,\RR)$.

In this section, we will investigate  the  problem of deep ReLU neural network approximation of functions in $L_2(\RRi, X, \gamma)$ or $L_2(\RRM, X, \gamma)$ with the  error measured in the norms of the  space $L_2(\RRi, X, \gamma)$  or of the space $L_\infty^{\sqrt{g}}(\RRM,X)$, respectively. (Notice that these norms  are the most important in evaluation of the error of collocation approximation of solutions of  parametric and stochastic PDEs). 
It is convenient to us to incorporate these different approximation  problems  into unified consideration. 
Hence, in what follows, we use   the joint notations: 
\begin{equation} \label{space:Ll(U,X)}
\Ll(U,X)
:=
\begin{cases}
L_{\infty}^{\sqrt{g}}( \RRM,X) \ \ &{\rm if} \ \ U=\RRM, \\
L_2(\RRi,X,\gamma) \ \ &{\rm if} \ \ U=\RRi;
\end{cases}
\end{equation}
\begin{equation} \label{set:Ff}
\Ff
:=
\begin{cases}
\NMO \ \ &{\rm if} \ \ U=\RRM, \\
\FF \ \ &{\rm if} \ \ U=\RRi;
\end{cases}
\end{equation}
and 
\begin{equation} \label{set:Nn}
\Nn
:=
\begin{cases}
\{1,...,M\} \ \ &{\rm if} \ \ U=\RRM, \\
\NN \ \ &{\rm if} \ \ U=\RRi.
\end{cases}
\end{equation}
Here $\FF$ is the set of all sequences of non-negative integers $\bs=(s_j)_{j \in \NN}$ such that their support 
$\supp (\bs):= \{j \in \NN: s_j >0\}$ is a finite set.

 Let $(H_k)_{k \in \NN_0}$ be the  Hermite polynomials normalized according to
$\int_{\RR} | H_k(y)|^2\, g(y)\, \rd y= 1.$ Then a function $v \in L_2(U,X,\gamma)$  can be represented  by the Hermite gpc expansion
\begin{equation} \label{series}
v(\by)=\sum_{\bs\in\Ff} v_\bs \,H_\bs(\by), \quad v_\bs \in X,
\end{equation}
with
\begin{equation}
H_\bs(\by)=\bigotimes_{j \in \Nn}H_{s_j}(y_j),\quad v_\bs:=\int_U v(\by)\,H_\bs(\by)\, \rd\gamma (\by), \quad \bs \in \Ff.
\label{hermite}
\end{equation}
Notice that $(H_\bs)_{\bs \in \Ff}$ is an orthonormal basis of $L_2(U,\gamma):= L_2(U,\RR, \gamma)$. 
 Moreover, for every $v \in L_2(U,X,\gamma)$  represented by the 
series \eqref{series},  Parseval's identity holds
\begin{equation} \nonumber
\|v\|_{L_2(U,X,\gamma)}^2
\ = \ \sum_{\bs\in\Ff} \|v_\bs\|_X^2.
\end{equation}

For $\bs, \bs' \in \Ff$,  the inequality $\bs' \le \bs$ means that $s_j' \le s_j$, $j \in \Nn$.   A set
$\bsigma=(\sigma_\bs)_{\bs \in \Ff}$ with $\sigma_\bs \in \RR$ is called increasing if $\sigma_{\bs'} \le \sigma_\bs$ for $\bs' \le \bs$. 

\medskip
\noindent
{\bf Assumption (I)} \ 	For $v \in L_2(U,X,\gamma)$  represented by 
the series \eqref{series}, there exists an increasing 
set  $\bsigma =(\sigma_\bs)_{\bs \in \Ff}$ of positive 
numbers  such that 
for some $q$ with $0< q < 2$,
\begin{equation} \label{sigma-summability}
\left(\sum_{\bs\in\Ff} (\sigma_\bs \|v_\bs\|_{X})^2\right)^{1/2} \ \le C_1 \ <\infty, \ \ \text{with} \ \
\norm{\bsigma^{-1}}{\ell_q(\Ff)} \le C_2 < \infty,
\end{equation}
where the constants $C_1$ and $C_2$  are independent of $J$.

Here and in what follows, "independent of $J$" means that $C_1$ and $C_2$ (and other constants)
are independent of $M$ when $J=M$, since we are interested in convergence rates and other  asymptotic properties which do not depend on $M$ and which is based on Assumption (I).

\begin{lemma}\label{lemma[AbsConv]}
	For $v \in L_2(U,X,\gamma)$ sastisfying Assumption {\rm (I)}, the 
series \eqref{series} converges absolutely and therefore, unconditionally in 
$\Ll(U,X)$  to $v$ and 
	\begin{equation} \label{summability1}
	\sum_{\bs\in\Ff} \|v_\bs\|_{X} 
	\ \le  \ C \ <\infty,
\end{equation}
where the constant $C$ is independent of $J$.
\end{lemma}

\begin{proof} By applying the H\"older inequality 
 from Assumption {\rm (I)}  we obtain
	\begin{equation} \nonumber
	\begin{split}	 
	\sum_{\bs\in\Ff} \|v_\bs\|_{X}
	\leq 
	\bigg( \sum_{\bs\in \Ff} (\sigma_\bs \|v_\bs\|_{X})^2\bigg)^{1/2} 
	\bigg(\sum_{\bs\in \Ff} \sigma_\bs^{-2} \bigg)^{1/2} 
	\le \
	C\norm{\bsigma^{-1}}{\ell_q(\Ff)}
	\ < \ \infty.
	\end{split}
	\end{equation}	
	This proves \eqref{summability1}. Hence, by the equality $\norm{H_\bs}{L_2(\RRi)} = 1$, $\bs \in \FF$,  and the inequality $\norm{H_\bs}{L_\infty^{\sqrt{g}}(\RRi)} < 1$, $\bs \in \NN_0^M$ (which follows from \eqref{CramerBnd} in Appendix), the series ~\eqref{series} converges absolutely, and therefore, unconditionally  to $v \in L_2(U,X,\gamma)$ since by the Parseval's identity  it already converges to $v$ in the norm of $L_2(U,X,\gamma)$.
	\hfill
\end{proof}

\subsection{ Sparse-grid Lagrange gpc interpolation}
\label{ Hermite gpc interpolation approximation}

For $m \in \NN_0$, let $Y_m = (y_{m;k})_{k \in \pi_m}$ be the increasing sequence of  the $m+1$ roots of the Hermite polynomial $H_{m+1}$, ordered as
$$
y_{m,-j} < \cdots < y_{m,-1} < y_{m,0} = 0 < y_{m,1} < \cdots < y_{m,j} \quad {\rm if} \  m = 2j,
$$
$$
y_{m,-j} < \cdots < y_{m,-1} < y_{m,1} < \cdots < y_{m,j} \quad {\rm if} \  m = 2j - 1,
$$
where 
$$
\pi_m:= 
\begin{cases}
\{-j,-j+1,..., -1, 0, 1, ...,j-1,j \} \ & \ \text{if} \  m = 2j; \\
\{-j,-j+1,...,-1, 1,...,j-1,j \} \ & \ \text{if} \  m = 2j-1.
\end{cases}
$$
(in particular, $Y_0 = (y_{0;0})$ with $y_{0;0} = 0$).

For  a function$v$  on $\RR$ taking values in a Hilbert space $X$ and $m \in \NN_0$, we define the   Lagrange intepolation operator $I_m$  by
\begin{equation} \label{I_(v)}
I_m(v):= \ \sum_{k\in \pi_m} v(y_{m;k}) L_{m;k}, \quad 
L_{m;k}(y) := \prod_{j \in \pi_m \ j\not=k}\frac{y - y_{m;j}}{y_{n;k} - y_{m;j}},
\end{equation}		
(in particular, $I_0(v) = v(y_{0,0})L_{0,0}(y)= v(0)$ and $L_{0,0}(y)=1$). Notice that $I_m(v)$ is a function on $\RR$ taking values in $X$ and  interpolating $v$ at $y_{m;k}$, i.e., $I_m(v)(y_{m;k}) = v(y_{m;k})$.   Moreover, for a function $v: \, \RR \to \RR$, the function $I_m(v)$ is the Lagrange polynomial having degree 
$\le m$, and that $I_m(\varphi) = \varphi$ for every polynomial $\varphi$ of degree $\le m$.

Let
\begin{equation} \label{lambda_m}
\lambda_m:= \ \sup_{\|v\|_{L_\infty^{\sqrt{g}}(\RR)} \le 1} \|I_m(v)\|_{L_\infty^{\sqrt{g}}(\RR)} 
\end{equation}	
be the Lebesgue constant. 
It was proven in \cite{Mat94a,Mat94b,Sza97} that
\begin{equation} \nonumber
\lambda_m
\ \le \
C(m+1)^{1/6}, \quad m \in \NN,
\end{equation}
for some positive constant $C$ independent of $m$ (with the obvious inequality $\lambda_0(Y_0) \le 1$). Hence, for every $\varepsilon > 0$, there exists a positive constant $C_\varepsilon \ge 1$ independent of $m$ such that
\begin{equation} \label{ineq[lambda_m]2}
\lambda_m
\ \le \
(1 + C_\varepsilon m)^{1/6 + \varepsilon}, \quad \forall m \in \NN_0.
\end{equation}

We define the univariate operator $\Delta_m$ for $m \in \NN_0$ by
\begin{equation} \nonumber
\Delta_m
:= \
I_m - I_{m-1},
\end{equation} 
with the convention $I_{-1} = 0$. 

\begin{lemma} \label{lemma[Delta_{bs}]}
	For every $\varepsilon > 0$, there exists a positive constant $C_\varepsilon$ independent of $m$ such that for every function  $v$ on $\RR$,
	\begin{equation} \label{ineq[Delta]1}
	\|\Delta_m(v)\|_{L_\infty^{\sqrt{g}}(\RR)}	
	\ \le \
	(1 + C_\varepsilon m)^{1/6 + \varepsilon} \|v\|_{L_\infty^{\sqrt{g}}(\RR)}	, \quad \forall m \in \NN_0,
	\end{equation}
	whenever the norm in the right-hand side is finite. 
\end{lemma}

\begin{proof} 
	From the assumptions we have that
	\begin{equation} \nonumber
	\|\Delta_m(v)\|_{L_\infty^{\sqrt{g}}(\RR)}	
	\ \le \
	2C(1 + m)^{1/6} \|v\|_{L_\infty^{\sqrt{g}}(\RR)}	, \quad \forall m \in \NN_0,
	\end{equation}
	which implies \eqref{ineq[Delta]1}.  
	\hfill
\end{proof}

We will use a sparse-grid Lagrange gpc interpolation as an intermediate approximation in the deep ReLU neural network approximation of functions $v \in L_2(U,X,\gamma)$. In order to have a correct definition of interpolation operator we have to impose some neccessary restrictions on $v$. Let $\Ee$ be a $\gamma$-measurable subset in $U$ such that $\gamma(\Ee) =1$ and $\Ee$ contains  all $\by \in U$ with $|\by|_0 < \infty$ in the case $U=\RRi$, where $|\by|_0$ denotes the number of nonzero components $y_j$ of $\by$. For a given $\Ee$ and Hilbert space $X$, we define $L_2^\Ee(U,X,\gamma)$ as the subspace in $L_2(U,X,\gamma)$  of all elements $v$ such that the point value $v(\by)$ (of a representative of $v$) is well-defined for all $\by \in \Ee$. In what folllows, $\Ee$ is fixed.

For  $v \in L_2^\Ee(U,X,\gamma)$, we introduce the tensor product operator $\Delta_\bs$, $\bs \in \Ff$, by
\begin{equation} \label{Delta_bs(v)}
\Delta_\bs(v)
:= \
\bigotimes_{j \in \Nn} \Delta_{s_j}(v),
\end{equation}
where the univariate operator
$\Delta_{s_j}$ is successively applied to the univariate function $\bigotimes_{i<j} \Delta_{s_i}(v)$ by considering it as a 
function of  variable $y_j$ with the other variables held fixed.
 From the definition of $L_2^\Ee(U,X,\gamma)$ one can see that the operators  $\Delta_\bs$ are well-defined for all $\bs \in \Ff$. We define for $\bs \in \Ff$,
\begin{equation} \nonumber
I_\bs(v)
:= \
\bigotimes_{j \in \Nn} I_{s_j}(v), \quad
L_{\bs;\bk}
:= \
\bigotimes_{j \in \Nn} L_{s_j;k_j}, \quad
\pi_\bs
:= \
\prod_{j \in \Nn} \pi_{s_j},
\end{equation}
(the function $I_\bs(v)$ is defined in the same manner as $\Delta_\bs(v)$).

For $\bs \in \Ff$ and $\bk \in \pi_\bs$, let $E_\bs$ be the subset in $\Ff$ of all $\be$ such that $e_j$ is either $1$ or $0$ if $s_j > 0$, and $e_j$ is $0$ if $s_j = 0$, and let $\by_{\bs;\bk}:= (y_{s_j;k_j})_{j \in \Nn} \in U$. Put $|\bs|_1 := \sum_{j \in \Nn} s_j$ for  $\bs \in \Ff$. It is easy to check that the interpolation operator $\Delta_\bs$ can be represented in the form
\begin{equation} \label{Delta_bs=}
\Delta_\bs(v)				
\ = \
\sum_{\be \in E_\bs} (-1)^{|\be|_1} I_{\bs - \be} (v)
\ = \
\sum_{\be \in E_\bs} (-1)^{|\be|_1} \sum_{\bk \in \pi_{\bs - \be}} v(\by_{\bs - \be;\bk}) L_{\bs - \be;\bk}.
\end{equation}

For a given finite set $\Lambda \subset \Ff$, we introduce the gpc interpolation operator $I_\Lambda$  by
\begin{equation} \label{I_Lambda}
I_\Lambda
:= \
\sum_{\bs \in \Lambda} \Delta_\bs.
\end{equation}
From \eqref{Delta_bs=} we obtain 
\begin{equation} \label{I_Lambda=}
I_\Lambda(v)				
\ = \
\sum_{\bs \in \Lambda} \sum_{\be \in E_\bs} (-1)^{|\be|_1} \sum_{\bk \in \pi_{\bs - \be}} v(\by_{\bs - \be;\bk}) L_{\bs - \be;\bk}.
\end{equation}
 
A set $\Lambda \subset \Ff$ is called downward closed if the inclusion $\bs \in \Lambda$ yields the inclusion $\bs' \in \Lambda$ for every $\bs' \in \Ff$ such that $\bs' \le \bs$. 

For $\theta, \lambda \ge 0$, we define the set $\bp(\theta, \lambda):= \brac{p_\bs(\theta, \lambda) }_{\bs \in \Ff}$ by 
\begin{equation} \label{[p_s]}
p_\bs(\theta, \lambda) := \prod_{j \in \Nn} (1 + \lambda s_j)^\theta, \quad \bs \in \Ff,
\end{equation}
with abbreviations $p_\bs(\theta):= p_\bs(\theta, 1)$ and $\bp(\theta):= \bp(\theta, 1)$.

Let $0 < q < \infty$ and $\bsigma= (\sigma_\bs)_{\bs \in \Ff}$ be a set of positive numbers. For $\xi >0$,  define the set
	\begin{equation} \label{Lambda(xi)}
\Lambda(\xi):= \{\bs \in \Ff: \, \sigma_\bs^q \le \xi\}.
\end{equation}

By the formula \eqref{I_Lambda=} we  can  represent the operator $I_{\Lambda(\xi)}$ in the form
\begin{equation} \label{I_Lambda(xi)=}
I_{\Lambda(\xi)}(v)				
\ = \
\sum_{(\bs,\be,\bk) \in G(\xi)}  (-1)^{|\be|_1} v(\by_{\bs - \be;\bk})L_{\bs - \be;\bk},
\end{equation}
where	
\begin{equation} \label{G(xi):=}
G(\xi)				
:= \
\{(\bs,\be,\bk) \in \Ff \times \Ff \times \Ff: \, \bs \in \Lambda(\xi), \ \be \in E_\bs, \ \bk \in \pi_{\bs - \be} \}.
\end{equation}

The following theorem gives an estimate for the error of the approximation of $v \in \Ll_2^\Ee(U,X,\gamma)$ by the sparse-grid Lagrange gpc interpolation $I_{\Lambda(\xi)} v$ on the sampling points in the set $G(\xi)$, which will be used in the deep ReLU neural approximation in the next section.

\begin{theorem}\label{lemma:coll-approx}
		Let $v \in \Ll_2^\Ee(U,X,\gamma)$ satisfy Assumption {\rm (I)}
	and let $\varepsilon >0$  be a fixed number. 
	Assume that $\norm{\bp(\theta/q,\lambda)\bsigma^{-1}}{\ell_q(\Ff)} \le C < \infty$, where 
	$\theta =7/3 + 2\varepsilon$, $\lambda:= C_\varepsilon$ is the constant in Lemma \ref{lemma[Delta_{bs}]}, and the constant $C$ is independent of $J$.
	Then for each $\xi > 1$, we have that 
	\begin{equation} \label{u-I_Lambdau, p le infty}
	\|v -I_{\Lambda(\xi)}v\|_{\Ll(U,X)} \leq C\xi^{-(1/q - 1/2)},
	\end{equation}
	where the  constant $C$ in \eqref{u-I_Lambdau, p le infty} is independent of $J$, $v$ and $\xi$.
\end{theorem}

A proof of this theorem is given in Appendix \ref{Proof of Theorem ref{lemma:coll-approx}}.
\begin{corollary}\label{corollary:coll-approx}
	Under the assumptions of Theorem \ref{lemma:coll-approx},  for each $n > 1$, we can construct a sequence of points  $Y_{\Lambda(\xi_n)}:= (\by_{\bs - \be;\bk})_{(\bs,\be,\bk) \in G(\xi_n)}$ so that
	$|G(\xi_n)| \le n$ and
	\begin{equation} \label{u-I_Lambda}
	\|v -I_{\Lambda(\xi_n)}v\|_{\Ll(U,X)} \leq Cn^{-(1/q - 1/2)},
	\end{equation}
	where the constant $C$ in \eqref{u-I_Lambda} is independent of $J$, $v$ and $n$.
\end{corollary}

\begin{proof}
Notice that this corollary was proven in \cite[Corollary 3.1]{Dung21} for the case $U=\RRi$. By Lemma~\ref{lemma:G(xi)<} in Appendix $|G(\xi)| \le C_q \xi$ for every $\xi > 1$. Hence, the corollary follows from Theorem~\ref{lemma:coll-approx}	by sellection of $\xi_n$ as the maximal number satisfying $|G(\xi_n)| \le n$. 
\hfill
\end{proof}

\subsection{Approximation by deep ReLU neural networks}

In this section, we construct deep ReLU neural networks for collocation approximation of functions $v \in L_2(U,X,\gamma)$. We primarily approximate 
{$v$} 
by the sparse-grid Lagrange gpc interpolation $I_{\Lambda(\xi)} v$. 
Under the assumptions of Lemma~\ref{lemma: |s|_1, |s|_0}(iii) in Appendix,  $I_{\Lambda(\xi)} v $  can be 
seen as a function on $\RRm$, where  
$
m := \min\left\{M,\lfloor K_q \xi \rfloor\right\}.
$
In the next step, we approximate 
$I_{\Lambda(\xi)} v $ by its truncation 
$I_{\Lambda(\xi)}^{\omega}v$ on a sufficiently large 
super-cube
\begin{align}
B^m_\omega
&
:=
[-2\sqrt{\omega}, 2\sqrt{\omega}]^m \subset \RR^m,
\label{B^m_omega}
\end{align}
where the parameter  $\omega$ depending on $\xi$ is chosen in an appropriate way. Finally, the function $I_{\Lambda(\xi)}^{\omega}v$ and therefore, $v$ is approximated by a function $\Phi_{\Lambda(\xi)}v $ on $\RRm$ which is constructed from a deep ReLU neural network. Let us discribe this construction.

For convenience, we consider $\RR^m$ as  the subset of all $\by \in U$ such that $y_j = 0$ for $j > m$.  If $f$ is a function on $\RRm$ taking values in a Hilbert space $X$, then $f$ has an extension to $\RR^{m'}$ with $m' > m$  and the whole $U$ which is denoted again by $f$, by the formula 
$f(\by)= f \brac{(y_j)_{j=0}^m}$ for $\by = (y_j)_{j =1}^{m'}$ and $\by = (y_j)_{j \in \Nn}$, respectively. 

Suppose that   deep ReLU   neural networks $\phi_{\bs - \be;\bk}$ on  $\RR^{|\supp(\bs)|}$ are already constructed for approximation of the polynomials $L_{\bs - \be;\bk}$, $(\bs,\be,\bk) \in G(\xi)$. Then the network $\bphi_{\Lambda(\xi)}:= (\phi_\bs)_{(\bs,\be,\bk) \in G(\xi)}$  on  $\RR^m$ with $|G(\xi)|$ outputs  which is constructed by parallelization,  is used to construct an approximation of
$I_{\Lambda(\xi)}^{\omega}v$ and hence of $v$.  Namely, we  approximate $v$ by 
\begin{equation} \label{Phi_v}
\Phi_{\Lambda(\xi)}v (\by) := \sum_{(\bs,\be,\bk) \in G(\xi)}  (-1)^{|\be|_1}  v(\by_{\bs - \be;\bk})\phi_{\bs - \be;\bk} (\by).
\end{equation}
For the set $\Lambda(\xi)$, we introduce  the following numbers:	
\begin{equation} \label{m_p(xi)}
m_1(\xi)
:= \ 
\max_{\bs \in \Lambda(\xi)} |\bs|_1,
\end{equation}		
and
\begin{equation} \label{m(xi)}
m(\xi)
:= \ 
\max\big\{j \in \Nn:  \exists \bs \in \Lambda(\xi)\ {\rm such \ that} \ s_j > 0 \big\}.
\end{equation}

In this section, we will prove our main results on deep ReLU neural network approximation of functions 
$v \in L_2^\Ee(U,X,\gamma)$ with the error measured in the norm of  the space $L_2(\RRi,X,\gamma)$ or of the space $L_\infty^{\sqrt{g}}(\RRM,X)$, which are incorporated into the following joint theorem.

Denote by $\be^i = (e^i_j)_{j \in \Nn}\in \Ff$ the element with $e^i_i = 1$ and $e^i_j = 0$ for $j \not= i$.
\begin{theorem}\label{thm1-dnn}
	Let   $v \in L_2^\Ee(U,X,\gamma)$ satisfy Assumption {\rm (I)}. Let $\theta$ be any number such that	$\theta \ge 3/q$.
Assume that 
the set $\bsigma= (\sigma_{\bs})_{\bs \in \Ff}$ in Assumption {\rm (I)} 
satisfies  $\sigma_{\be^{i'}} \le \sigma_{\be^i}$ if $i' 
< i$, 
	and that $\norm{\bp(\theta)\bsigma^{-1}}{\ell_q(\Ff)} \le C < \infty$, where the constant $C$ is independent of $J$. Let $K_q$, $K_{q,\theta}$ and $C_q$ be the constants in the assumptions of   Lemma \ref{lemma: |s|_1, |s|_0} and of Lemma \ref{lemma:G(xi)<} in Appendix.
Then for every $\xi > 2$, we can  construct  a deep 
ReLU neural network $\bphi_{\Lambda(\xi)}:= (\phi_{\bs - \be;\bk})_{(\bs,\be,\bk) \in G(\xi)}$ on 
$\RR^m$ with 
\begin{equation} \nonumber
m
:=
\begin{cases}
\min\left\{M,\lfloor K_q \xi \rfloor\right\}  \ \ &{\rm if} \ \ U=\RRM, \\
\lfloor K_q \xi \rfloor \ \ &{\rm if} \ \ U=\RRi,
\end{cases}
\end{equation}
and a sequence of points  $Y_{\Lambda(\xi)}:= (\by_{\bs - \be;\bk})_{(\bs,\be,\bk) \in G(\xi)}$ having the following properties. 
\begin{itemize}
	\item [{\rm (i)}]
The deep ReLU neural network $\bphi_{\Lambda(\xi)}$ and sequence of points  $Y_{\Lambda(\xi)}$ are independent of $v$;
	
\item [{\rm (ii)}] 
The  output dimension of $\bphi_{\Lambda(\xi)}$ are at most $\lfloor C_q \xi \rfloor$;
	
\item [{\rm (iii)}] 
$W\big(\bphi_{\Lambda(\xi)}\big) \le C \xi^{1 + 2/\theta q} \log \xi$;

\item [{\rm (iv)}] 
$L\big(\bphi_{\Lambda(\xi)}\big) \le C \xi^{1/\theta q} (\log \xi)^2$;

\item [{\rm (v)}] 
The  components $\phi_{\bs - \be;\bk}$, $(\bs,\be,\bk) \in G(\xi)$, of $\bphi_{\Lambda(\xi)}$ are deep ReLU neural networks on 
$\RR^{|\supp(\bs)|}$ with ${|\supp(\bs)|}  \le K_{q,\theta} \xi^{\frac{1}{\theta q}}$, having support contained in the super-cube $[-T,T]^{|\supp(\bs)|}$, 
where  $T:= 4\sqrt{\lfloor K_{q,\theta} \xi \rfloor}$;

\item [{\rm (vi)}] The approximation of $v$ by  $\Phi_{\Lambda(\xi)}v$  gives the error estimate
\begin{align} \label{approximation-error}
\| v- \Phi_{\Lambda(\xi)}v  \|_{\Ll(U,X)}\leq C\xi^{-(1/q - 1/2)}.
\end{align} 
\end{itemize}
Here the constants $C$  are independent of $J$, $v$ and $\xi$.
\end{theorem}

Let us briefly draw a plan  of the proof of this theorem.
We will give a detailed proof for the case $U=\RRi$ and then point out that  the case $U= \RRM$ can be proven in the same  way with slight modification. 

 In what follows in this 
 section, all  definitions, formulas  and assertions are given for the case $U= \RRi$, and for $\xi >1$, we use the letters $m$ and 
 $\omega$ only for the notations 
 \begin{align} \label{m,omega}
 m :=\lfloor K_q\xi \rfloor, \quad \omega := \lfloor {K_{q,\theta}}\xi \rfloor,
 \end{align}
 where $K_q$ and $K_{q,\theta}$ are the constants defined   
 in Lemma  \ref{lemma: |s|_1, |s|_0} in Appendix. 
As mentioned above, we primarily approximate 
$v \in L_2(\RRi,X,\gamma)$ 
by the gpc interpolation $I_{\Lambda(\xi)} v$.    In the next step, we approximate 
$I_{\Lambda(\xi)} v $ by its truncation 
$I_{\Lambda(\xi)}^{\omega}v$ on the 
super-cube $ B^m_\omega$, which will be constructed below.
  The final step is to construct a deep ReLU  neural network  $\bphi_{\Lambda(\xi)}:= (\phi_{\bs - \be;\bk})_{(\bs,\be,\bk) \in G(\xi)}$ to approximate 
$I_{\Lambda(\xi)}^{\omega}v$ by $\Phi_{\Lambda(\xi)}v$ of the form \eqref{Phi_v}.

 For a function $\varphi$ defined on $\RR$, 
we denote by $\varphi^{\omega}$ the truncation of $\varphi$ on $B^1_\omega$, i.e.,
\begin{align}
\varphi^{\omega}(y)
:=
\begin{cases}
\varphi(y)
&
\text{if }
y \in B^1_\omega
\\
0
&
\text{otherwise}.
\label{truncation}
\end{cases}
\end{align}
If $\supp (\bs) \subset \{1,...,m\}$, we 
put
$$
L_{\bs,\bk}^{\omega}(\boldsymbol{y})
:=
\prod_{j=1}^m L_{s_j;k_j}^{\omega}(y_j),\qquad \by\in \RR^m. 
$$
We have
$
L_{\bs,\bk}^{\omega}(\boldsymbol{y})
=
\prod_{j=1}^m L_{s_j;k_j}(y_j)
$
if 
$
\boldsymbol{y}\in B^m_\omega $, and $L_{\bs,\bk}^{\omega}(\boldsymbol{y})=0$ otherwise.
For a function $v \in L_2^\Ee(\RRi,X,\gamma)$,
we define
\begin{equation} \label{def:I_Lambda(xi)^omega}
I_{\Lambda(\xi)}^\omega(v)				
:= 
\sum_{(\bs,\be,\bk) \in G(\xi)}(-1)^{|\be|_1} v(\by_{\bs - \be;\bk}) L_{\bs - \be;\bk}^\omega.
\end{equation}

Let the assumptions of Theorem 
\ref{thm1-dnn} hold. By Lemma \ref{lemma: |s|_1, |s|_0}(iii) in Appendix for every $\xi >2$ we have  $m(\xi) \le m$. Hence,  for every $(\bs,\be,\bk) \in G(\xi)$, $L_{\bs - \be;\bk}$ and  $L_{\bs - \be;\bk}^\omega$ and therefore, $I_{\Lambda(\xi)}v $ and $I_{\Lambda(\xi)}^{\omega} v$ can be considered as functions on $\RRm$.  For $g \in L_2(\RRm,X,\gamma)$, we have $\norm{g}{ L_2(\RRm,X,\gamma)} = \norm{g}{ L_2(\RRi,X,\gamma)}$ in the sense of extension of $g$. We will make use of these facts  without mention.

To prove Theorem \ref{thm1-dnn} we will use some intermediate approximations for estimation of the approximation error as in \eqref{approximation-error}. Suppose that the deep ReLU  neural network  $\bphi_{\Lambda(\xi)}$ and therefore, the function $\Phi_{\Lambda(\xi)}$ are already constructed.  By the triangle inequality we have 
\begin{equation}\label{intermediate-approx}
\begin{aligned}
\|v- \Phi_{\Lambda(\xi)} v\|_{L_2(\RRi,X,\gamma)}
& \leq \|v-I_{\Lambda(\xi)}v  \|_{L_2(\RRi,X,\gamma)}
+ \|I_{\Lambda(\xi)}v - I_{\Lambda(\xi)}^{\omega}v\|_{{L_2(\RRm \setminus B^m_\omega,X,\gamma)}}
\\
& 
+ \| I_{\Lambda(\xi)}^{\omega}v- \Phi_{\Lambda(\xi)} v\|_{L_2(B^m_\omega,X,\gamma)}
+ \| \Phi_{\Lambda(\xi)} v  \|_{L_2(\RRm \setminus B^m_\omega,X,\gamma)}.
\end{aligned}
\end{equation}
Hence the estimate \eqref{approximation-error} will be done via the bound $C\xi^{-(1/q - 1/2)}$ for every of the four terms in the right-hand side. The first term is  already  estimated as in Theorem \ref{lemma:coll-approx}.  The estimates for the others  will  be carried out in the following lemmata 
(Lemmata \ref{lemma:S-S^omega}--\ref{lemma:G_Lambda}). To complete the proof of Theorem \ref{thm1-dnn} we have  also to prove the bounds of the size and depth of  $\bphi_{\Lambda(\xi)}$ according to the items (iii) and (iv) which  are given  in Lemma \ref{lemma:size} below.

	For $v \in L_2^\Ee(\RRi,X,\gamma)$ satisfying Assumption {\rm (I)}, by Lemma \ref{lemma[AbsConv]} the 
series \eqref{series} converges  unconditionally in 
$L_2(\RRi,X,\gamma)$  to $v$.  Therefore, the formula \eqref{I_Lambda=} for $\Lambda = \Lambda(\xi)$ can be rewritten as
\begin{equation} \label{I_Lambda(xi)=1}
I_{\Lambda(\xi)}(v)				
\ = \
\sum_{\bs \in \Lambda(\xi)} \sum_{\bs' \in \FF}v_{\bs'} \sum_{\be \in E_\bs} (-1)^{|\be|_1} 
\sum_{\bk \in \pi_{\bs - \be}} H_{\bs'}(\by_{\bs - \be;\bk}) L_{\bs - \be;\bk}.
\end{equation}
Hence,  we also have by the definition \eqref{def:I_Lambda(xi)^omega}
\begin{equation} \label{I_Lambda(xi)^omega=1}
I_{\Lambda(\xi)}^\omega(v)				
= 
\sum_{\bs \in \Lambda(\xi)} \sum_{\bs' \in \FF}v_{\bs'} \sum_{\be \in E_\bs} (-1)^{|\be|_1} 
\sum_{\bk \in \pi_{\bs - \be}} H_{\bs'}(\by_{\bs - \be;\bk}) L_{\bs - \be;\bk}^\omega.
\end{equation}  
 
\begin{lemma}\label{lemma:S-S^omega}
		Under the assumptions of Theorem 
\ref{thm1-dnn}, for every $\xi > 1$, we have that 
\begin{align}
\norm{I_{\Lambda(\xi)}v - I_{\Lambda(\xi)}^{\omega} v}{L_2(\RRi,X,\gamma)}
&
\leq
C\xi^{-(1/q - 1/2)},
\label{I_Lambda - I_Lambda^omega}
\end{align}
where the constant $C$  is independent of $v$ and $\xi$.
\end{lemma}

\begin{proof} 
	By  the equality 
		\begin{align}
	\norm{L_{\bs - \be;\bk} - L_{\bs - \be;\bk}^\omega}{L_2(\RRi,\gamma)}
	& = \ \norm{L_{\bs - \be;\bk}}{L_2(\RR^m\setminus B^m_\omega,\gamma)},
	\quad \forall (\bs,\be,\bk) \in G(\xi),
	\notag
\end{align}
 and the triangle inequality, noting \eqref{I_Lambda(xi)=1} and \eqref{I_Lambda(xi)^omega=1},  we obtain 
	\begin{align}
	\norm{I_{\Lambda(\xi)}v  -  I_{\Lambda(\xi)}^{\omega}v}{L_2(\RRi,X,\gamma)}
	& 
	\leq 
	\sum_{\bs \in \Lambda(\xi)} \sum_{\bs' \in \FF}	\norm{v_{\bs'}}{X} \sum_{\be \in E_\bs}
	\sum_{\bk \in \pi_{\bs - \be}} |H_{\bs'}(\by_{\bs - \be;\bk})|\norm{L_{\bs - \be;\bk}}{L_2(\RR^m\setminus B^m_\omega,\gamma)}.
	\notag
	\end{align}
	Let  $(\bs,\be,\bk) \in G(\xi)$ be given. Then we have
	$$
	L_{\bs - \be;\bk}
	=
	\prod_{j =1}^m  L_{s_j - e_j;k_j}(y_j), \quad \by \in \RR^m,
	$$ 
	where $L_{s_j - e_j;k_j}$ is a polynomial in variable $y_j$, 
of degree not greater than $m_1(\xi) $.
Hence, applying Lemma~\ref{l:Rm trun1}  in Appendix with taking account of \eqref{m,omega} gives
	\begin{align}
\norm{L_{\bs - \be;\bk}}{L_2(\RR^m\setminus B^m_\omega,\gamma)}
	\le 
C\xi
e^{- K_1\xi}
\norm{L_{\bs - \be;\bk}}{L_2(\RR^m,\gamma)}.
\notag
\end{align}
From Lemmas \ref{lemma:sumH_s'} and \ref{lemma: normL_{s;k}<} and Lemma \ref{lemma: |s|_1, |s|_0}(ii)  in Appendix we derive that 
	\begin{align*}
\norm{L_{\bs - \be;\bk}}{L_2(\RR^m,\gamma)}
&=
 \prod_{j \in \NN} \norm{L_{s_j - e_j;k_j}}{L_2(\RR,\gamma)}
 \le 
  \prod_{j \in \NN} e^{K_2 (s_j - e_j)}
  \\
 &
  \le 
  \prod_{j \in \NN} e^{K_2 s_j }
 = e^{K_2 |\bs|_1}
 \le 
 e^{K_2 m_1(\xi)}
  \le e^{K_3\xi^{1/\theta q}},
\notag
\end{align*}
and
\begin{align} 
\sum_{\bk \in \pi_{\bs - \be}} |H_{\bs'}(\by_{\bs - \be;\bk})| 
\le  e^{K_4|\bs|_1}
 \le 
e^{K_4 m_1(\xi)}
  \le e^{K_5\xi^{1/\theta q}}.
\label{H_{bs'}<}
\end{align}
Summing up,  we arrive at
	\begin{align}
\norm{I_{\Lambda(\xi)}v  -  I_{\Lambda(\xi)}^{\omega}v}{L_2(\RRi,X,\gamma)}
& 
\leq 
C_1\xi
\exp\Big(- K_1\xi + (K_2 + K_5) \xi^{1/\theta q}\Big)
\sum_{\bs \in \Lambda(\xi)} \sum_{\bs' \in \FF}	\norm{v_{\bs'}}{X} \sum_{\be \in E_\bs} 1 
\notag
\\ & 
\le 
C_1 \xi \exp\Big(- K_1\xi + K_6 \xi^{1/\theta q}\Big) 
|G(\xi)|  \sum_{\bs' \in \FF}	\norm{v_{\bs'}}{X}.
\notag
\end{align}
Hence, by Lemma \ref{lemma[AbsConv]}, Lemma \ref{lemma:G(xi)<}   in Appendix and the inequality $1/\theta q \le 1/3$ we get 
	\begin{align}
\norm{I_{\Lambda(\xi)}v  -  I_{\Lambda(\xi)}^{\omega}v}{L_2(\RRi,X,\gamma)}
\le 
C_2 \xi^2 \exp\Big(- K_1\xi + K_6 \xi^{1/\theta q}\Big) 
\le 
C\xi^{-(1/q - 1/2)}.
\notag
\end{align}
	\hfill
\end{proof}

The previous lemma gives  the bound of the second term in  the right-hand side of \eqref{intermediate-approx}, i.e., the error bound for the approximation of sparse-grid Lagrange interpolation $I_{\Lambda(\xi)}v$ by its truncation  $I_{\Lambda(\xi)}^{\omega}v$ on $B_m^\omega$ for $v \in L_2(\RRi,X,\gamma)$. As the next step, we will construct a deep ReLU neural network  $\bphi_{\Lambda(\xi)}:= \brac{\phi_{\bs - \be;\bk}}_{(\bs,\be,\bk) \in G(\xi)}$ on $\RRm$ for approximating $I_{\Lambda(\xi)}^{\omega}v$ by the function $\Phi_{\Lambda(\xi)}v$ given  as in \eqref{Phi_v}, and prove the bound of the error as the third term in  the right-hand side of \eqref{intermediate-approx}.

For $s\in \NN_0$, we represent  the univariate  interpolation polynomial $L_{s;k}$ in the form of linear combination of monomials:
\begin{align}\label{L_s}
 L_{s;k}(y) =:	\sum_{\ell=0}^s b^{s;k}_\ell y^\ell.  
\end{align} 
From \eqref{L_s} for each  $(\bs,\be,\bk) \in G(\xi)$  we have 
\begin{align}
L_{\bs - \be;\bk}
=
\sum_{\bell=\boldsymbol{0}}^{\bs - \be}  b^{\bs - \be;\bk}_{\bell} \by^{\bell},
\label{L_{bs-be}}
\end{align}
where the summation $\sum_{\bell=\boldsymbol{0}}^{\bs - \be}$ means that the sum is taken over  all $\bell$ such that 
$\boldsymbol{0} \le \bell \le {\bs - \be}$, and 
$$
b^{\bs - \be;\bk}_{\bell} = \prod_{j=1}^m  b^{s_j - e_j;k_j}_{\ell_j}, \quad 
\by^{\bell}= \prod_{j=1}^m y_j^{\ell_j}.
$$
Indeed,  we have 
\begin{align}
L_{\bs - \be;\bk}
&=
\prod_{j =1}^m  L_{s_j - e_j;k_j}(y_j)
=
\prod_{j =1}^m\sum_{\ell_j=0}^{s_j-e_j} b^{s_j-e_j;k_j}_{\ell_j} y_j^{\ell_j}
\notag
\\
&=
\sum_{\bell=\boldsymbol{0}}^{\bs - \be} 
\brac{\prod_{j=1}^m  b^{s_j - e_j;k_j}_{\ell_j}} \by^{\bell} 
= 
\sum_{\bell=\boldsymbol{0}}^{\bs - \be}  b^{\bs - \be;\bk}_{\bell} \by^{\bell}.
\notag
\end{align}
By \eqref{I_Lambda(xi)^omega=1}  and \eqref{L_{bs-be}} we get for every $\by \in B^m_\omega$,
	\begin{equation} \label{I_Lambda(xi)^omega=2}
I_{\Lambda(\xi)}^\omega(v)(\by)				
= 
\sum_{(\bs,\be,\bk) \in G(\xi)}  (-1)^{|\be|_1} v(\by_{\bs - \be;\bk})
\sum_{\bell=\boldsymbol{0}}^{\bs - \be}  b^{\bs - \be;\bk}_{\bell}
\brac{2\sqrt{\omega}}^{|\bell|_1} \prod_{j \in \supp(\bell)}\brac{\frac{y_j}{2\sqrt{\omega}}}^{\ell_j}.
\end{equation}

%
%
Let
$\bell \in \FF$ be such that  $\boldsymbol{0} \le \bell  \le 
	\bs - \be$. By defintion we have $\supp(\bell) \subset \supp(\bs) $. By changing variables 
	$$
	\bx= \frac{\by}{2\sqrt{\omega}}, \ \ \bx \in \supp(\bs),
	$$ 
	we have
	\begin{equation} 
	\prod_{j \in \supp(\bell)} \left(\frac{y_j}{2\sqrt{\omega}}\right)^{\ell_j}
	= 
	\prod_{j \in \supp(\bell)} 
	\varphi_1^{\ell_j}\brac{\frac{y_j}{2\sqrt{\omega}}} \prod_{j \in \supp(\bs) \setminus\supp(\bell)} 
	\varphi_0\brac{\frac{y_j}{2\sqrt{\omega}}}=h^{\bs - \be}_{\bell}(\bx), 
	\ \ \by \in B^{|\supp(\bs)|}_\omega,
\end{equation}
where
	\begin{equation} \label{varphi=}
	h^{\bs - \be}_{\bell}(\bx) 
		:=
		\prod_{j \in \supp(\bell)} \varphi_1^{\ell_j}(x_j)
		\prod_{j \in \supp(\bs)\setminus\supp(\bell)}
		\varphi_0\brac{x_j}, 
\end{equation}
$\varphi_0$ and $\varphi_1$ are the piece-wise linear functions defined before Lemma~\ref{lem-product-varphi}.
We put
\begin{align}\label{B-bs:=}
B_\bs:=  
\displaystyle
\max_{\be \in E_\bs, \, \bk \in \pi_{\bs - \be}} \ \max_{\boldsymbol{0}\leq \bell\le \bs- \be} \big|b^{\bs - \be;\bk}_{\bell}\big|,
\end{align}
and
\begin{align}\label{delta^{-1}:=}
\delta^{-1}
:= 
\xi^{1/q - 1/2} \sum_{\bs \in \Lambda(\xi)} e^{K|\bs|_1} p_\bs(2)
\brac{2\sqrt{\omega}}^{|\bs|_1} B_\bs, 
\end{align}
where $K$ is the constant in Lemma \ref{lemma:sumH_s'} in Appendix.
Hence, by applying Lemma \ref{lem-product-varphi} to the product in  the left-hand side of \eqref{varphi=}, for every  $(\bs,\be,\bk) \in G(\xi)$  and $\bell$ 
satisfying  $\boldsymbol{0} < \bell  \le \bs-\be$,
there exists a deep  ReLU neural network
$\phi^{\bs - \be}_{\bell}$ on $\RR^{|\supp(\bs)|}$ with 
$\supp \big(\phi^{\bs - \be}_{\bell}\big) \subset [-2,2]^{|\supp(\bs)|}$
such that 
\begin{align}\label{<delta} 	
\sup_{\by \in B^{|\supp(\bs)|}_{\omega}}\abs{\prod_{j \in \supp(\bs)}  
	\brac{\frac{y_j}{2\sqrt{\omega}}}^{\ell_j} 	-	\phi^{\bs - \be}_{\bell}\brac{ \frac{\by}{\sqrt{\omega}}}}
\ \le \
\sup_{\by \in B^{|\supp(\bs)|}_{4\omega}}\abs{
	h^{\bs - \be}_{\bell}\brac{ \frac{\by}{2\sqrt{\omega}}} -	\phi^{\bs - \be}_{\bell}\brac{ \frac{\by}{2\sqrt{\omega}}}}
\ \le \
\delta,
\end{align}
and
\begin{align}\label{supp(g^ell_s)} 
\supp\bigg(\phi^{\bs - \be}_{\bell}\brac{\frac{\cdot}{2\sqrt{\omega}}}\bigg) \subset B^{|\supp(\bs)|}_{4\omega}. 
\end{align}
Also, from Lemma \ref{lem-product-varphi} and the inequalities 
$|\bell|_1 + |\supp(\bs)\setminus\supp(\bell)| \le |\bs|_1 \leq \delta^{-1}$ one can see that
\begin{align} \label{W(phi_s,ell)}
W\brac{\phi^{\bs - \be}_{\bell}}
&
\leq
C
\brac{ 
	1 
	+ 
	\abs{\bs}_1 
	\brac{ 
		\log\abs{\bs}_1
		+
		\log  
		\delta^{-1}
	}
} 
\leq
C
\brac{ 
	1 
	+ 
	\abs{\bs}_1   
	\log  
	\delta^{-1}
}
\end{align}
and
\begin{align} \label{L(phi_s,ell)}
L\brac{\phi^{\bs - \be}_{\bell}}
\leq
C
\brac{ 1 + \log \abs{\bs}_1 \brac{\log\abs{\bs}_1 + \log \delta^{-1}}} 
\leq
C
\brac{ 1 + \log	\abs{\bs}_1   \log \delta^{-1}}.
\end{align}

We define the deep ReLU neural network $\phi_{\bs - \be;\bk}$ on  $\RR^{|\supp (\bs)|}$ by
\begin{align} \label{g_Lambda^*}
\phi_{\bs - \be;\bk}(\by)
&
:=
\sum_{\bell=\boldsymbol{0}}^{\bs - \be} 
b^{\bs - \be;\bk}_{\bell}
\brac{2\sqrt{\omega}}^{|\bell|_1}
\phi^{\bs - \be}_{\bell}
\brac{ \frac{\by}{2\sqrt{\omega}}},
\quad \by \in \RR^{|\supp (\bs)|},
\end{align}
which is the  parallelization deep ReLU neural network of the component deep ReLU neural networks 
$\phi^{\bs - \be}_{\bell}\brac{\frac{\cdot}{2\sqrt{\omega}}}$. 
From \eqref{supp(g^ell_s)} it follows
\begin{align}\label{supp(g)} 
\supp \brac{\phi_{\bs - \be;\bk}} \subset B^{|\supp (\bs)|}_{4\omega}.
\end{align}

According to the above convention,  for $(\bs,\be,\bk) \in G(\xi)$, in some places  without mention we identify  the functions $\phi^{\bs - \be}_{\bell}$  and $\phi_{\bs - \be;\bk}$ on $\RR^{|\supp (\bs)|}$  with their extentions  on $\RRm$ or on $\RRi$ due to the inclusions  $\supp(\bs) \subset \{1,...,m\} \subset \NN$. 

We define $\bphi_{\Lambda(\xi)}:= \brac{\phi_{\bs - \be;\bk}}_{(\bs,\be,\bk) \in G(\xi)}$ as the deep ReLU neural network on $\RRm$ which is realized by 
parallelization of $\phi_{\bs - \be;\bk}, \ (\bs,\be,\bk) \in G(\xi)$. Consider 
the approximation of $I_{\Lambda(\xi)}^{\omega}v$ by the function $\Phi_{\Lambda(\xi)} v$ where for convenience  we recall 
\begin{equation} \label{Phi_Lambda(xi):=}
\Phi_{\Lambda(\xi)} v(\by) := \sum_{(\bs,\be,\bk) \in G(\xi)} (-1)^{|\be|_1}  v(\by_{\bs - \be;\bk}) \phi_{\bs - \be;\bk}(\by).
\end{equation}

\begin{lemma}\label{lemma:I-g}
	Under the assumptions of Theorem \ref{thm1-dnn}, for every $\xi > 1$, we have
	\begin{align}
	\norm{I_{\Lambda(\xi)}^{\omega} v - \Phi_{\Lambda(\xi)} u}{L_2(B^m_\omega ,X,\gamma)}
	& \leq 
	C \xi^{-(1/q - 1/2)},
	\label{err flambda xi}
	\end{align}  
	where the constant $C$ is independent of $v$ and $\xi$.
\end{lemma}

\begin{proof}
	According to Lemma \ref{lemma[AbsConv]} the series \eqref{series} converges uncondionally to $v$. Hence,
	for every $\by \in B^m_\omega$, we have by \eqref{I_Lambda(xi)^omega=1} 
	\begin{equation} \label{I_Lambda(xi)^omega=3}
	I_{\Lambda(\xi)}^\omega(v) (\by)				
	= 
	\sum_{\bs \in \Lambda(\xi)} \sum_{\bs' \in \FF}v_{\bs'} \sum_{\be \in E_\bs} (-1)^{|\be|_1} 
	\sum_{\bk \in \pi_{\bs - \be}} H_{\bs'}(\by_{\bs - \be;\bk})
	\sum_{\bell=\boldsymbol{0}}^{\bs - \be}  b^{\bs - \be;\bk}_{\bell}
	\brac{2\sqrt{\omega}}^{|\bell|_1} 
	\prod_{j \in \supp(\bs)}\brac{\frac{y_j}{2\sqrt{\omega}}}^{\ell_j},
	\end{equation}
	and by  \eqref{Phi_Lambda(xi):=}
	\begin{equation} \label{Phi_Lambda(xi)=}
\Phi_{\Lambda(\xi)} v (\by)				
= 
\sum_{\bs \in \Lambda(\xi)} \sum_{\bs' \in \FF}v_{\bs'} \sum_{\be \in E_\bs} (-1)^{|\be|_1} 
\sum_{\bk \in \pi_{\bs - \be}} H_{\bs'}(\by_{\bs - \be;\bk})
\sum_{\bell=\boldsymbol{0}}^{\bs - \be}  b^{\bs - \be;\bk}_{\bell} \brac{2\sqrt{\omega}}^{|\bell|_1}
\phi^{\bs - \be}_{\bell}
\brac{\frac{\by}{2\sqrt{\omega}}}.
\end{equation}		
From these formulas and \eqref{<delta} we derive the inequality
	\begin{align} \label{norm{I_Lambda}1}
	\norm{I_{\Lambda(\xi)}^{\omega}v - \Phi_{\Lambda(\xi)} v}{L_2(B^m_\omega,X,\gamma)}
	& \leq 
	\sum_{\bs \in \Lambda(\xi)} \sum_{\bs' \in \FF}\norm{v_{\bs'}}{X} \sum_{\be \in E_\bs} 
		\sum_{\bk \in \pi_{\bs - \be}} |H_{\bs'}(\by_{\bs - \be;\bk})|
		\sum_{\bell=\boldsymbol{0}}^{\bs - \be}  \abs{b^{\bs - \be;\bk}_{\bell}}  \brac{2\sqrt{\omega}}^{|\bell|_1}\delta.
	\end{align}
We have by \eqref{B-bs:=}
		\begin{align*}
	\sum_{\bell=\boldsymbol{0}}^{\bs - \be}  \abs{b^{\bs - \be;\bk}_{\bell}}
	\le B_\bs \prod_{j \in \support(\bs - \be)} s_j
	\le  p_\bs(1) B_\bs,
	\end{align*}
and by Lemma \ref{lemma:sumH_s'}  in Appendix 		
		\begin{align} \label{sum_{b in E,k in pi}}
	\sum_{\be \in E_\bs} 
	\sum_{\bk \in \pi_{\bs - \be}} |H_{\bs'}(\by_{\bs - \be;\bk})|
	\le 
		\sum_{\be \in E_\bs} 	e^{K|\bs-\be|_1}
 \le 
 2^{|\bs|_0}e^{K|\bs|_1}
 \le 
  p_\bs(1) e^{K|\bs|_1}.	
\end{align}
This together with  \eqref{norm{I_Lambda}1}, Lemma \ref{lemma[AbsConv]} and \eqref{delta^{-1}:=} yields	that
		\begin{align}
	\norm{I_{\Lambda(\xi)}^{\omega}v - \Phi_{\Lambda(\xi)} v}{L_2(B^m_\omega,X,\gamma)}
	&\leq 
	\notag
	\sum_{\bs \in \Lambda(\xi)} \  \delta B_\bs p_\bs(1)\sum_{\bs' \in \FF}\norm{v_{\bs'}}{X} \brac{2\sqrt{\omega}}^{|\bs|_1}
	\sum_{\be \in E_\bs} 
	\sum_{\bk \in \pi_{\bs - \be}} |H_{\bs'}(\by_{\bs - \be;\bk})|
	\\
	&\leq 
	\notag
\sum_{\bs' \in \FF}\norm{v_{\bs'}}{X} \ \delta	\sum_{\bs \in \Lambda(\xi)} e^{K|\bs|_1}p_\bs(2)\brac{2\sqrt{\omega}}^{|\bs|_1} B_\bs 
\\
&\le 
C\xi^{-(1/q - 1/2)}.
\notag	
	\end{align}
	\hfill
\end{proof}

In the previous lemma, we   proved  the  bound of the third term in  the right-hand side of \eqref{intermediate-approx}, i.e., the error bound for the approximation of  $I_{\Lambda(\xi)}^{\omega}v$  by the function $\Phi_{\Lambda(\xi)}v$  for $v \in L_2(\RRi,X,\gamma)$. As the last step in the error estimation, we will establish the bound for the fourth term in the right-hand side of \eqref{intermediate-approx}.

\begin{lemma}\label{lemma:G_Lambda}
	Under the assumptions of Theorem \ref{thm1-dnn}, for every $\xi > 1$, we have	
	\begin{align}
	\norm{\Phi_{\Lambda(\xi)} v}{L_2((\RRm \setminus B^m_\omega) ,X,\gamma)}
	& \leq 
	C\xi^{-(1/q - 1/2)},
	\label{norm-g_Lambda}
	\end{align}  
	where the constant $C$ is independent of $v$ and $\xi$.
\end{lemma}

\begin{proof} We use the formula \eqref{Phi_Lambda(xi)=} to estimate the norm 
	$\norm{\Phi_{\Lambda(\xi)} v}{L_2((\RRm \setminus B^m_\omega) ,X,\gamma)}$. We need the following auxiliary inequality
	\begin{align}\label{phi^{bs - be;bk}_{bell}<}
	\left|\phi^{\bs - \be}_{\bell}
	\left( 
	\frac{\by}{2\sqrt{\omega}}\right)\right| \leq 2, \ \forall \by \in \RRm.
	\end{align}
Due to \eqref{supp(g^ell_s)}, it is sufficient to prove this inequality for 
 $\by \in B^{|\supp(\bs)|}_{4\omega}$.
	Considering   the right-hand side of  \eqref{delta^{-1}:=}, we have 
	\begin{align}\label{eq-Absbell-2}
	\sum_{\bs \in \Lambda(\xi)} e^{K|\bs|_1} p_\bs(2)
	\brac{2\sqrt{\omega}}^{|\bs|_1} B_\bs
	\	\ge  \
	e^{K|{\boldsymbol{0}}|_1} p_{\boldsymbol{0}}(2)
	\brac{2\sqrt{\omega}}^{|{\boldsymbol{0}}|_1} B_{\boldsymbol{0}}
	\	= \ 1.
	\end{align}
	With the definition \eqref{delta^{-1}:=}, this yields that  $\delta  \le 1$.
	On the other hand, by the definition of $h^{\bs - \be}_{\bell}$,
	\begin{align}\nonumber
		\sup_{\by \in B^{|\supp(\bs)|}_{4\omega}}	\left|h^{\bs - \be}_{\bell}
		\left( 
		\frac{\by}{2\sqrt{\omega}}\right)\right| 
		\ \le \
		1
	\end{align}
	From the last two  inequalities, \eqref{<delta} and the triangle inequality  we derive
	\eqref{phi^{bs - be;bk}_{bell}<} for $\by \in B^{|\supp(\bs)|}_{4\omega}$.
	
	By \eqref{phi^{bs - be;bk}_{bell}<} and Lemma \ref{l:Rm trun1} in Appendix,
	\begin{align*} 
	\norm{\phi^{\bs - \be}_{\bell} \brac{ \frac{\cdot}{2\sqrt{\omega}}}}{L_2(\RRm \setminus B^m_\omega,\gamma)}
	\le
	2	\norm{1}{L_2(\RRm \setminus B^m_\omega,\gamma)}
	\le	  
		C_1 m \exp \left(- K_1\omega \right).
\end{align*}	
	This together with  \eqref{Phi_Lambda(xi)=} implies that
	\begin{align*} 
	&\norm{\Phi_{\Lambda(\xi)} v}{L_2(\RRm \setminus B^m_\omega ,X,\gamma)}
	\\
	& 
	\leq
	\sum_{\bs \in \Lambda(\xi)} \sum_{\bs' \in \FF} 	\|v_{\bs'}\|_X\sum_{\be \in E_\bs} 
	\sum_{\bk \in \pi_{\bs - \be}} \abs{H_{\bs'}(\by_{\bs - \be;\bk})}
	\sum_{\bell=\boldsymbol{0}}^{\bs - \be}  \abs{b^{\bs - \be;\bk}_{\bell}} \brac{2\sqrt{\omega}}^{|\bell|_1}
	\norm{\phi^{\bs - \be}_{\bell} \brac{ \frac{\cdot}{2\sqrt{\omega}}}}{L_2(\RRm \setminus B^m_\omega,\gamma)}
	\\
	&\le
	C_1m \exp \left(- K_1\omega \right)	\sum_{\bs \in \Lambda(\xi)} \brac{2\sqrt{\omega}}^{|\bs|_1}\sum_{\bs' \in \FF} 	\|v_{\bs'}\|_X\sum_{\be \in E_\bs} 
	\sum_{\bk \in \pi_{\bs - \be}} \abs{H_{\bs'}(\by_{\bs - \be;\bk})}
	\sum_{\bell=\boldsymbol{0}}^{\bs - \be}  \abs{b^{\bs - \be;\bk}_{\bell}}. 
	\end{align*}
	By a  tensor product argument from Lemma \ref{lemma: sum |b_ell|}  in Appendix and the inequality 
	$\bs - \be \le \bs$ for $\be \in E_\bs$, we deduce  the estimates
	\begin{align} \label{sumb^{bs - be;bk}}
	\sum_{\bell=\boldsymbol{0}}^{\bs - \be}  \abs{b^{\bs - \be;\bk}_{\bell}}
	\le 
	e^{K_2|\bs|_1} \bs!
	\le 
	e^{K_2|\bs|_1} |\bs|_1^{|\bs|_1}, 
	\end{align}	
	which and \eqref{sum_{b in E,k in pi}} give
		\begin{align} \label{sum_{s in E}}
	\sum_{\bk \in \pi_{\bs - \be}} \abs{H_{\bs'}(\by_{\bs - \be;\bk})}
	\sum_{\bell=\boldsymbol{0}}^{\bs - \be}  \abs{b^{\bs - \be;\bk}_{\bell}}
	\le 
	\sum_{\bk \in \pi_{\bs - \be}} \abs{H_{\bs'}(\by_{\bs - \be;\bk})}	e^{K_2|\bs|_1} |\bs|_1^{|\bs|_1} 
	\le 
p_\bs(1)	e^{K_2|\bs|_1} |\bs|_1^{|\bs|_1}. 
	\end{align}	
This  in combining  with \eqref{m,omega}, \eqref{H_{bs'}<}, Lemma \ref{lemma[AbsConv]}  allows us to continue the estimation as
		\begin{equation} \label{norm-estimates}
		\begin{aligned}	
	\norm{\Phi_{\Lambda(\xi)} v}{L_2(\RRm \setminus B^m_\omega ,X,\gamma)}
	& \leq 
C_1m \exp \left(- K_1\omega \right)	\sum_{\bs' \in \FF} 	\|v_{\bs'}\|_X \sum_{\bs \in \Lambda(\xi)} \brac{2\sqrt{\omega}}^{|\bs|_1} p_\bs(1)	e^{K_2|\bs|_1} |\bs|_1^{|\bs|_1} 
	\\
	&\leq 
	C_2m \exp \left(- K_1\omega \right)	\sum_{\bs \in \Lambda(\xi)} \brac{2\sqrt{\omega}}^{|\bs|_1} p_\bs(1)	e^{K_2|\bs|_1} |\bs|_1^{|\bs|_1} 
		\\
	&\leq 
	C_2\xi \exp \left(- K_1 \xi \right) \brac{C_3 \xi^{1/2}}^{m_1(\xi)} e^{K_2{m_1(\xi)}}  \big[m_1(\xi)\big]^{m_1(\xi)}	
	\sum_{\bs \in \Lambda(\xi)} p_\bs(1).
	\end{aligned}
	\end{equation}
By the assumption of Theorem \ref{thm1-dnn} $\norm{\bp(\theta)\bsigma^{-1}}{\ell_q(\Ff)} \le C < \infty$ for some $\theta \ge 3/q$,  we derive that
$$
\norm{\bp(3/q)\bsigma^{-1}}{\ell_q(\Ff)} \le 
\norm{\bp(\theta)\bsigma^{-1}}{\ell_q(\Ff)} \le C < \infty.
$$
Applying Lemma \ref{lemma: |s|_1, |s|_0}(i) in Appendix  gives
	\begin{align*} 
	\sum_{\bs \in \Lambda(\xi)} p_\bs(1) \le \sum_{\bs \in \Lambda(\xi)} p_\bs(3) \le C. 
\end{align*}
 Hence by \eqref{norm-estimates} and Lemma \ref{lemma: |s|_1, |s|_0}(ii)  in Appendix we have that
	\begin{align*} 
	\norm{\Phi_{\Lambda(\xi)} v}{L_2(\RRm \setminus B^m_\omega ,X,\gamma)}
		&\leq 
	C_2\xi \exp \left(- K_1 \xi\right) \brac{C_3 \xi^{1/2}}^{K_{q,\theta} \xi^{1/\theta q}} e^{K_3K_{q,\theta} \xi^{1/\theta q}}  \big(K_{q,\theta} \xi^{1/\theta q}\big)^{K_{q,\theta} \xi^{1/\theta q}}	
C_4 \xi	
\\
&\leq 
C_5  \xi^2 \exp(- K_1 \xi + K_4\xi^{1/\theta q} \log \xi + K_5 \xi^{1/\theta q}). 
	\end{align*}
	Since $1/\theta q \le  1/3$, we obtain
	\begin{align*} 
	\norm{\Phi_{\Lambda(\xi)} v}{L_2(\RRm \setminus B^m_\omega ,X,\gamma)}
	& \leq
	 C \xi^{-(1/q -1/2)}.
		\end{align*}
	\hfill
\end{proof}

To complete the proof of Theorem \ref{thm1-dnn}, we have to establish the bounds of the size and depth of the deep ReLU neural network $\bphi_{\Lambda(\xi)}$ as in (iii) and (iv).

\begin{lemma}\label{lemma:size}
	Under the assumptions of Theorem \ref{thm1-dnn}, 
	the input  dimension of 
	$\bphi_{\Lambda(\xi)}$ is at most $\lfloor K_q \xi \rfloor$, for every $\xi > 1$, the  output dimension of 
	$\bphi_{\Lambda(\xi)}$ at most $\lfloor C_q \xi \rfloor$, 
	\begin{align}
	W\big(\bphi_{\Lambda(\xi)} \big)
	\le
	C  \xi^{1 + 2/\theta q} \log \xi,
	\label{W} 
	\end{align}
	and
	\begin{align}
	L\big(\bphi_{\Lambda(\xi)} \big)
	\le 
C \xi^{1/\theta q} (\log \xi)^2,
	\label{L}
	\end{align}
	where the constants $C$ are independent of $v$ and $\xi$.
\end{lemma}

\begin{proof} 
	The input dimension of $\bphi_{\Lambda(\xi)}$ is not greater than $m(\xi)$ which is at most $\lfloor K_q \xi \rfloor$
	by Lemma~\ref{lemma: |s|_1, |s|_0}(iii) in Appendix.	
	The output dimension of $\bphi_{\Lambda(\xi)}$ is the number $|G(\xi)|$ which is at most $\lfloor C_q \xi \rfloor$
	by Lemma~\ref{lemma:G(xi)<} in Appendix.
	
	By  Lemmas \ref{lem:parallel}  and \ref{lem-product-varphi} and \eqref{W(phi_s,ell)} 
	the size of  $\bphi_{\Lambda(\xi)} $ is estimated as
	\begin{align} \label{W<}
	W\big(\bphi_{\Lambda(\xi)} \big)
	&\le
	\sum_{(\bs,\be,\bk) \in G(\xi)} W\brac{\phi_{\bs - \be;\bk}}
	\le 
	\sum_{\bs \in \Lambda(\xi)} \sum_{\be \in E_\bs}
	\sum_{\bk \in \pi_{\bs - \be}} 
	\sum_{\bell=\boldsymbol{0}}^{\bs - \be}  
	W\brac{\phi^{\bs - \be}_{\bell}}
	\\
	&\leq 
	C_1
\sum_{\bs \in \Lambda(\xi)} \sum_{\be \in E_\bs}
\sum_{\bk \in \pi_{\bs - \be}} 
\sum_{\bell=\boldsymbol{0}}^{\bs - \be}  
\brac{ 1 + \abs{\bs}_1 \log \delta^{-1}},
	\label{W<2}
	\end{align}	
where we recall,	
	\begin{align*}
\delta^{-1}
:= 
\xi^{1/q - 1/2} \sum_{\bs \in \Lambda(\xi)} e^{K_1|\bs|_1} p_\bs(2)
\brac{2\sqrt{\omega}}^{|\bs|_1} B_\bs, 
\end{align*}
	\begin{align*}
	B_\bs:=  
	\displaystyle
	\max_{\be \in E_\bs, \, \bk \in \pi_{\bs - \be}} \ \max_{\boldsymbol{0}\leq \bell\le \bs- \be} \big|b^{\bs - \be;\bk}_{\bell}\big|.
	\end{align*}
	From \eqref{sumb^{bs - be;bk}} it follows that
		\begin{align*}
			B_\bs
			\le 
			\max_{\be \in E_\bs, \, \bk \in \pi_{\bs - \be}}	 
	\sum_{\bell=\boldsymbol{0}}^{\bs - \be}  \abs{b^{\bs - \be;\bk}_{\bell}}
 \le \exp\brac{K_2 \xi^{1/\theta q} \log \xi},
	\end{align*}
	which by Lemma \ref{lemma: |s|_1, |s|_0}(i)	 in Appendix implies
		\begin{align*}
	\delta^{-1}
	&\le
	\xi^{1/q - 1/2}\exp\brac{K_2 \xi^{1/\theta q} \log \xi} \sum_{\bs \in \Lambda(\xi)}  p_\bs(2)
	\\&
	\le
C_2\xi^{1/q +1/2}\exp\brac{K_3 \xi^{1/\theta q} \log \xi} 
\le C_2\exp\brac{K_3 \xi^{1/\theta q} \log \xi}. 
	\end{align*}
Hence,
	\begin{align} \label{log delta}
	\log(\delta^{-1})
	&
	\le 
	K_4 \xi^{1/\theta q} \log \xi.
	\end{align}
and consequently,
	\begin{align*} 
\brac{ 1 + \abs{\bs}_1 \log \delta^{-1}}
&
\le 
\brac{ 1 + \abs{\bs}_1 K_4 \xi^{1/\theta q} \log \xi}
\le
C_2 \xi^{2/\theta q} \log \xi.
\end{align*}
From \eqref{W<}--\eqref{W<2} and Lemma \ref{lemma:G(xi)<} in Appendix we obtain the desired bound of the size of $\bphi_{\Lambda(\xi)}$:
	\begin{align*}
W\big(\bphi_{\Lambda(\xi)} \big)
&\leq 
C_2  \xi^{2/\theta q} \log \xi
\sum_{\bs \in \Lambda(\xi)} \sum_{\be \in E_\bs}
\sum_{\bk \in \pi_{\bs - \be}} 
\sum_{\bell=\boldsymbol{0}}^{\bs - \be}  1
\\
&\le 
C_2  \xi^{2/\theta q} \log \xi
\sum_{ (\bs,\be,\bk) \in G(\xi)}p_{\bs}(1)
\le 
C_3  \xi^{1 + 2/\theta q} \log \xi.
\end{align*}	

	By  using Lemma \ref{lem:parallel}, \eqref{L(phi_s,ell)}, \eqref{log delta} and  Lemma \ref{lemma: |s|_1, |s|_0}(ii) in Appendix, we prove that  the depth of  $\bphi_{\Lambda(\xi)} $ is bounded as in \eqref{L}:
\begin{align} 
L\brac{\bphi_{\Lambda(\xi)}}
&\le
\max_{(\bs,\be,\bk) \in G(\xi)} L\brac{\phi_{\bs - \be;\bk}}
\le 
\max_{(\bs,\be,\bk) \in G(\xi)}
\max_{\boldsymbol{0}\le \bell \le \bs - \be}  
L\brac{\phi^{\bs - \be}_{\bell}}
\notag
\\
&\leq 
C_4
\max_{(\bs,\be,\bk) \in G(\xi)}
\max_{\boldsymbol{0}\le \bell \le \bs - \be}   
\brac{ 1 + \log	\abs{\bs}_1   \log \delta^{-1}}
\notag
\\
&\leq 
C_4
\max_{\bs \in \Lambda(\xi)}  
\brac{ 1 + \log	\abs{\bs}_1   \log \delta^{-1}}
\notag
\\
&\leq 
C_4
\max_{\bs \in \Lambda(\xi)}  
\brac{ 1 + \log	 \brac{K_{q,\theta} \xi^{1/\theta q}} \,  	\brac{K_5 \xi^{1/\theta q} \log \xi}}
\le 
C_5 \xi^{1/\theta q} (\log \xi)^2.
\notag
\end{align}	
	\hfill
\end{proof}

We are now in a position to give a formal proof of Theorem \ref{thm1-dnn}.

\begin{proof} [Proofs of Theorem \ref{thm1-dnn}] \  From \eqref{intermediate-approx}, Theorem \ref{lemma:coll-approx} and Lemmata \ref{lemma:S-S^omega} -- \ref{lemma:G_Lambda}, for every $\xi > 2$,  we deduce that
\begin{align*} \label{v-G}
\norm{v - \Phi_{\Lambda(\xi)} v}{L_2(\RRi,X,\gamma)}
\leq 
C
\xi^{- (1/q - 1/2)}.
\end{align*}
The claim (vi) is proven.  The claim (i) follows directly from the construction of the deep ReLU neural network $\bphi_{\Lambda(\xi)}$ and the sequence of points $Y_{\Lambda(\xi)}$, the claim (ii) from Lemma \ref{lemma:G(xi)<}, the claims (iii)--(iv)  from Lemma \ref{lemma:size}  and  the claim (v) from Lemma \ref{lemma: |s|_1, |s|_0}(ii)  in Appendix and \eqref{supp(g)}. Thus, Theorem \ref{thm1-dnn} is proven for the case when $U=\RRi$. 

The case $U= \RRM$ can be proven in the same  way with a slight modification.  Counterparts of all  definitions, formulas  and assertions which have been used in the proof of the case $U= \RRi$, are true  for the case $U= \RRM$. In the proof of this case, in parlicular, the used equality $\norm{H_\bs}{L_2(\RRi)} = 1$, $\bs \in \FF$, is replaced by the inequality $\norm{H_\bs}{L_\infty^{\sqrt{g}}(\RRM)} < 1$, $\bs \in \NN_0^M$.
\hfill
\end{proof}

\section{Application to parametrized elliptic PDEs}
\label {lognormal inputs}

In this section, we apply the results in the previous section to the deep ReLU neural network approximation of  the solution $u$ to the parametrized elliptic PDEs \eqref{p-ellip} with lognormal inputs \eqref{lognormal}.
This is  based on the weighted $\ell_2$-summability of  the series $(\|u_\bs\|_V)_{\bs \in \Ff}$ in  following lemma which has been proven in \cite[Theorems 3.3 and 4.2]{BCDM17}.

\begin{lemma}\label {lemma[ell_2summability]}
	Assume that  there exist a number 
	$0<q<\infty$ and an increasing sequence $\brho =(\rho_{j}) 
	_{j \in \Nn}$ of   numbers strictly larger than 1 
	such that $\norm{\brho^{-1}}{\ell_q(\Nn)} \le C < \infty$  and
	\begin{equation} \nonumber
	\left\| \sum _{j \in \Nn} \rho_j |\psi_j| \right\|_{L_\infty(D)}
	\le C <\infty,
	\end{equation} 
	where the constants $C$ are independent of $J$.
	Then we have that for any $\eta \in \Nn$,
	\begin{equation} \label{sigma_s}
	\sum_{\bs\in\Ff} (\sigma_{\bs} \|u_\bs\|_V)^2 \le C < \infty \quad \text{with} \quad 
	\sigma_{\bs}^2:=\sum_{\|\bs'\|_{\ell_\infty(\Ff)}\leq \eta}{\bs\choose \bs'} \prod_{j \in \Nn}\rho_j^{2s_j'},
	\end{equation}
		where the constant $C$ is independent of $J$.
\end{lemma}

The following two lemmata are proven in \cite[Lemmata 5.2 and 5.3]{Dung21}.
\begin{lemma} \label{lemma:measurable}
	Let the assumptions of Lemma~\ref {lemma[ell_2summability]} hold.  Then the solution map 
	$\by \mapsto u(\by)$ is $\gamma$-measurable and $u \in L_2(U,V,\gamma)$.	 Moreover, $u \in L_2^\Ee(U,V,\gamma)$ where
	\begin{equation} \label{Ee}
	\Ee
	:= \
	\left\{\by \in \RRi: \  \sup_{j \in \NN} \rho_j^{-1} |y_j| < \infty \right\}
	\end{equation}
	having $\gamma(\Ee) =1$ and  containing  all $\by \in \RRi$ with $|\by|_0 < \infty$ in the case when $U=\RRi$.
\end{lemma}

\begin{lemma} \label{lemma[bcdm]}
	Let $0 < q <\infty$, 
	$\brho=(\rho_j) _{j \in \Nn}$ be a sequence of  positive 
	numbers such that $\norm{\brho^{-1}}{\ell_q(\Nn)} \le C < \infty$, where the constant $C$ is independent of $J$. Let $\theta$ be an 
	arbitrary nonnegative number and  $\bp(\theta)=(p_\bs(\theta))_{\bs \in \Ff}$ the set given as in \eqref{[p_s]}. For 
	$\eta \in \NN$,  let the set $\bsigma=(\sigma_\bs)_{\bs \in \Ff}$ be 
	defined as in \eqref{sigma_s}.
	Then for any  $\eta > \frac{2(\theta + 1)}{q}$, we have
	\begin{equation} \nonumber
	\norm{\bp(\theta)\bsigma^{-1}}{\ell_q(\Ff)} \le C < \infty, 
	\end{equation}
	where the constant $C$ is independent of $J$.
\end{lemma}

We are now in position to formulate our main results on collocation deep ReLU neural network approximation of the solution $u$ to parametric  elliptic PDEs with lognormal inputs. 

\begin{theorem}\label{thm-PDE-lognormal-dnn}
	Under the assumptions of Lemma 
	\ref{lemma[ell_2summability]},  let $0 < q < 2$. Then, given
	an arbitrary number $\delta  > 0$,  
	for every integer $n > 2$, we can  construct  a deep ReLU neural network 
	$\bphi_{\Lambda(\xi_n)}:= (\phi_{\bs - \be;\bk})_{(\bs,\be,\bk) \in G(\xi_n)}$ of the size
	$W\big(\bphi_{\Lambda(\xi_n)}\big) \le n$ 
	on  $\RR^m$ with 
	\begin{equation} \nonumber
	m
	:=
	\begin{cases}
	\min \left\{M,\left\lfloor K \brac{\frac{n}{\log n}}^{\frac{1}{1+\delta}} \right\rfloor \right\}  \ \ &{\rm if} \ \ U=\RRM, \\
	\left\lfloor K \brac{\frac{n}{\log n}}^{\frac{1}{1+\delta}} \right\rfloor \ \ &{\rm if} \ \ U=\RRi,
	\end{cases}
	\end{equation}	
and a sequence of points  $Y_{\Lambda(\xi_n)}:= (\by_{\bs - \be;\bk})_{(\bs,\be,\bk) \in G(\xi_n)}$	having the following properties. 
	\begin{itemize}
		\item [{\rm (i)}]
		The deep ReLU neural network $\bphi_{\Lambda(\xi_n)}$ and sequence of points  $Y_{\Lambda(\xi_n)}$ are  independent of $u$;
		\item [{\rm (ii)}] 
		The output dimension of $\bphi_{\Lambda(\xi_n)}$ is at most
		 $\left\lfloor K \brac{\frac{n}{\log n}}^{\frac{1}{1+\delta}} \right\rfloor $;	
		\item [{\rm (iii)}] 
		$L\big(\bphi_{\Lambda(\xi_n)}\big)\le C_\delta \left(\frac{n}{\log n}\right)^{\frac{\delta}{2(1 + \delta)}}\brac{\log n}^2$;
		\item [{\rm (iv)}] 
		The  components $\phi_{\bs - \be;\bk}$, $(\bs,\be,\bk) \in G(\xi_n)$, of $\bphi_{\Lambda(\xi_n)}$ are deep ReLU neural networks on 
		$\RR^{m_\bs}$ with $m_\bs  \le C_\delta n^\delta$, having support contained in the super-cube $[-T,T]^{m_\bs}$, where
		  $T:= C_\delta  \brac{\frac{n}{\log n}}^{\frac{1}{2(1+\delta)}} $;
		\item [{\rm (v)}]
		The approximation of $u$ by  $\Phi_{\Lambda(\xi_n)}u$ defined as in \eqref{Phi_v},  gives the error estimate 
		$$
		\| u- \Phi_{\Lambda(\xi_n)} u \|_{\Ll(U,V)} \leq C \left(\frac{n}{\log n}\right)^{- \frac{1}{1 + \delta}\brac{\frac{1}{q} - \frac{1}{2}}}.
		$$
	\end{itemize}
	Here the constants $C$, $K$ and  $C_\delta$ are independent of $J$, $u$ and $n$.	
\end{theorem}

\begin{proof} To prove the theorem we apply Theorem \ref{thm1-dnn} to the solution $u$. Without loss of generality we can assume that $\delta \le 1/6$. We take first the number 
	$\theta := 2/\delta q$ satisfying the inequality $\theta \ge 3/q$, and then choose a number $\eta \in \NN$ satisfying the inequality $\eta > 2(\theta + 1)/q$. By using Lemmata \ref{lemma[ell_2summability]}--\ref{lemma[bcdm]}, one can check that   $u \in L_2^\Ee(U,V,\gamma)$ satisfies the assumptions of Theorem \ref{thm1-dnn} for $X= V$ and the set  $(\sigma_\bs)_{\bs \in \FF}$ defined as in \eqref{sigma_s}, where $\Ee$ is the set defined in Lemma \ref{lemma:measurable}. For a given integer $n > 2$, we choose $\xi_n >2$ as the maximal number satisfying the inequality $C \xi_n^{1 + \delta} \log \xi_n \le n$, where $C$ is  the constant in the claim (ii) of Theorem~\ref{thm1-dnn}. It is easy to verify that there exist positive constants $C_1$ and $C_2$ independent of $n$ such that 
	$$
	C_1\left(\frac{n}{\log n}\right)^{\frac{1}{1 + \delta}} \le  \xi_n \le C_2\left(\frac{n}{\log n}\right)^{\frac{1}{1 + \delta}}.
	$$ 
	From Theorem \ref{thm1-dnn} with $\xi=\xi_n$ we deduce the desired results.
	\hfill
\end{proof}

From Theorem \ref{thm-PDE-lognormal-dnn} one can directly derive the following

\begin{theorem}\label{thm:PDE-lognormal-dnn}
	Under the assumptions of Lemma 
	\ref{lemma[ell_2summability]},  let $0 < q < 2$ and $\delta_q:= \min \brac{1, 1/q -1/2}$. Then, given
	an arbitrary number $\delta  \in (0,\delta_q)$,   
	for every integer $n > 1$, we can  construct  a deep ReLU neural network 
	$\bphi_{\Lambda(\xi_n)}:= (\phi_{\bs - \be;\bk})_{(\bs,\be,\bk) \in G(\xi_n)}$
	of the size $W\big(\bphi_{\Lambda(\xi_n)}\big) \le n$ 
	on $\RR^m$ with 
		\begin{equation} \nonumber
	m
	:=
	\begin{cases}
	\min \left\{M, \left\lfloor K n^{1 - \delta} \right\rfloor\right\}  \ \ &{\rm if} \ \ U=\RRM, \\
	\left\lfloor K n^{1 - \delta} \right\rfloor \ \ &{\rm if} \ \ U=\RRi,
	\end{cases}
	\end{equation}	
	and a sequence of points  $Y_{\Lambda(\xi_n)}:= (\by_{\bs - \be;\bk})_{(\bs,\be,\bk) \in G(\xi_n)}$ having the following properties. 
	\begin{itemize}
		\item [{\rm (i)}]
		The deep ReLU neural network $\bphi_{\Lambda(\xi_n)}$ and sequence of points  $Y_{\Lambda(\xi_n)}$ are  independent of $u$;
		\item [{\rm (ii)}] 
		The output dimension of $\bphi_{\Lambda(\xi_n)}$ are at most $\left\lfloor K n^{1 - \delta} \right\rfloor$;	
		\item [{\rm (iii)}] 
		$L\big(\bphi_{\Lambda(\xi_n)}\big)\le C_\delta n^\delta$;
		\item [{\rm (iv)}]
			The  components $\phi_{\bs - \be;\bk}$, $(\bs,\be,\bk) \in G(\xi_n)$, of $\bphi_{\Lambda(\xi_n)}$ are deep ReLU neural networks on 
		$\RR^{m_\bs}$ with $m_\bs  \le C_\delta n^\delta$, having support contained in the super-cube $[-T,T]^{m_\bs}$, 
		where  $T:= C_\delta n^{1 - \delta}$;
		\item [{\rm (v)}]
		The approximation of $u$ by  $\Phi_{\Lambda(\xi_n)}u$ defined as in \eqref{Phi_v},  gives the error estimates 
		\begin{align} \label{error esitimates}
		\| u- \Phi_{\Lambda(\xi_n)} u \|_{\Ll(U,V)} 
		\leq 
		C m^{- \brac{\frac{1}{q} - \frac{1}{2}}}
		 \leq 
		 C_\delta n^{- (1 - \delta)\brac{\frac{1}{q} - \frac{1}{2}}}.
	\end{align}
	\end{itemize}
	Here the constants $K$, $C$ and $C_\delta$  are independent of $J$, $u$ and $n$.	
\end{theorem}

Let us compare the collocation approximation of $u$ by the function
\begin{equation} \label{Phi_Lambda u}
\Phi_{\Lambda(\xi_n)}u := \sum_{(\bs,\be,\bk) \in G(\xi_n)}  (-1)^{|\be|_1}  u(\by_{\bs - \be;\bk})\phi_{\bs - \be;\bk},
\end{equation}
generated from the deep ReLU neural network $\bphi_{\Lambda(\xi_n)}$
as in Theorem \ref{thm:PDE-lognormal-dnn}, and 
the   collocation approximation of $u$ by  the sparse-grid Lagrange gpc interpolation  
\begin{equation} \label{I_Lambda u}
I_{\Lambda(\xi_n)}u  := \sum_{(\bs,\be,\bk) \in G(\xi_n)}  (-1)^{|\be|_1}  u(\by_{\bs - \be;\bk})L_{\bs - \be;\bk}.
\end{equation}
Both the methods are  based on $m$ the same particular solvers $\brac{u(\by_{\bs - \be;\bk})}_{(\bs,\be,\bk) \in G(\xi_n)}$.
From Corollary~\ref{corollary:coll-approx} one can see that under the assumptions of Theorem \ref{thm:PDE-lognormal-dnn}, there holds the  error bound in $m$ for the last approximation:
$$
\norm{u- I_{\Lambda(\xi_n)} u}{\Ll(U, V)} \le C m^{- \brac{\frac{1}{q} - \frac{1}{2}}},
$$
which is the same as that in  \eqref{error esitimates} for the first approximation since by the construction the parameter $m$ in  \eqref{error esitimates} can be treated as independent.

After the present paper and the paper  \cite{DNP21} appeared in ArXiv website, we have been informed about the paper \cite{SZ2021} on some problems similar to the problems considered in \cite{DNP21}  in a private communication with its authors.

\bigskip
\noindent
{\bf Acknowledgments.}  This work  is funded by Vietnam National Foundation for Science and Technology Development (NAFOSTED) under  Grant No. 102.01-2020.03. A part of this work was done when  the author was working at the Vietnam Institute for Advanced Study in Mathematics (VIASM). He would like to thank  the VIASM  for providing a fruitful research environment and working condition.

\appendix
\section{Appendix} \label{appendix}

\subsection{Auxiliary lemmata}

\begin{lemma}\label{lemma: |s|_1, |s|_0}
	Let $\theta\geq 0$ and $0<q<\infty$.
Let $\bsigma= (\sigma_{\bs})_{\bs \in \Ff}$  be a set of numbers strictly larger than $1$. Then  we have the following.
	\begin{itemize}
		\item[{\rm (i)}]
			If $\norm{\bp\brac{\frac{\theta}{q}}\bsigma^{-1}}{\ell_q(\Ff)} \le K < \infty$, where the constant $K$ is independent of $J$, then 
		\begin{align}
		\sum_{ \bs\in \Lambda(\xi)}
		p_{\bs}(\theta) 
		\leq
		K \xi \quad \forall \xi > 1.
		\label{sum_Lambda p_s}
		\end{align}
		In particular, 	if 	$\norm{\bsigma^{-1}}{\ell_q(\Ff)}^q \le K_q < \infty$, where the constant 
		$K_q \ge 1$ and is independent of $J$, then
		the set $\Lambda(\xi)$ is finite and 
		\begin{align} \label{|Lambda(xi)|}
		|\Lambda(\xi)|
		\le K_q \xi\quad \forall \xi > 1.
		\end{align}
		\item[{\rm (ii)}]
		If $\norm{\bp(\theta)\bsigma^{-1}}{\ell_q(\Ff)}^{1/\theta} \le  K_{q,\theta} < \infty$, where the constant $K_{q,\theta}$ is independent of $J$, then 
		\begin{align} \label{|s|_1}
		m_1(\xi)
		\le K_{q,\theta} \xi^{\frac{1}{\theta q}}\quad \forall \xi > 1.
		\end{align}
			\item[{\rm (iii)}]
		If $\sigma_{\be^{i'}} \le \sigma_{\be^i}$ for $i' < i$, and if 
		$\norm{\bsigma^{-1}}{\ell_q(\Ff)}^q \le K_q < \infty$, where the constant $K_q \ge 1$  and is independent of $J$, then 
		\begin{align} \label{m(xi)<}
		m(\xi)
		\le K_q \xi \quad \forall \xi > 1.
		\end{align}
	\end{itemize}
\end{lemma}

\begin{proof} The claim (ii) and (iii) were proven in \cite[Lemmata 3.2 and 3.3]{DNP21} for the case $\Ff = \FF$. The case $\Ff = \NN_0^M$ can be proven in a similar way.  Let us prove the claim (i).
	Indeed, we have for every $\xi > 1$,
	\begin{align}
	\sum_{\bs \in \Lambda(\xi)} 	p_{\bs}(\theta) 
	\le
	\sum_{\bs \in \Ff: \ \sigma_{\bs}^{-q}\xi\ge 1} p_{\bs}(\theta) 	\xi \sigma_{\bs}^{-q}
	\notag
	\leq 
	\xi
	\sum_{\bs\in \Ff}
	p_{\bs}(\theta)
	\sigma_{\bs}^{-q}
	\leq 
	C\xi.
	\notag
	\end{align}
	\hfill
\end{proof}

\begin{lemma}\label{lemma:G(xi)<} Let $\theta\geq 0$, $0<q<\infty$ and $\xi > 1$. Let $\bsigma= (\sigma_{\bs})_{\bs \in \Ff}$  be a set of numbers strictly larger than $1$. 
	If 	 $\norm{\bp\brac{\frac{\theta+2}{q}}\bsigma^{-1}}{\ell_q(\Ff)} \le C < \infty$, where the constant $C$ is independent of $J$, then there holds 
	\begin{align}
	\sum_{ (\bs,\be,\bk) \in G(\xi)}
	p_{\bs}(\theta) 
	\leq
	C\xi \quad \forall \xi > 1.
	\label{ps sigmas1}
	\end{align}
	In particular, 	if 	and $\norm{\bp\brac{\frac{2}{q}}\bsigma^{-1}}{\ell_q(\Ff)}^q \le C_q < \infty$, where the constant $C$ is independent of $J$, then
	$$
	|G(\xi)| \leq C_q \xi \quad \forall \xi > 1.
	$$
\end{lemma}

\begin{proof}
	We have for every $\xi > 1$,
	\begin{align}
	\sum_{ (\bs,\be,\bk) \in G(\xi)} p_{\bs}(\theta) 
	&
	=
	\sum_{\bs \in \Lambda(\xi)}  \sum_{\be \in E_\bs}
	\sum_{\bk \in \pi_{\bs - \be}} 
	p_{\bs}(\theta) 
	\le 
	\sum_{\bs \in \Lambda(\xi)} 	p_{\bs}(\theta)  \sum_{\be \in E_\bs} |\pi_{\bs - \be}|
	\\	
	&\le
	\sum_{\bs \in \Lambda(\xi)} 	p_{\bs}(\theta)  |E_\bs| 	p_{\bs}(1)
	=
	\sum_{\bs \in \Lambda(\xi)} 	p_{\bs}(\theta + 1)  2^{|\bs|_0}
	\le
	\sum_{\bs \in \Lambda(\xi)} 	p_{\bs}(\theta + 2)
	\\	
	&\le  
	\sum_{\bs \in \Ff: \ \sigma_{\bs}^{-q}\xi\ge 1} p_{\bs}(\theta + 2) 	\xi \sigma_{\bs}^{-q}
	\notag
	\leq 
	\xi
	\sum_{\bs\in \Ff}
	p_{\bs}(\theta+2)
	\sigma_{\bs}^{-q}
	\leq 
	C\xi.
	\notag
	\end{align}
	\hfill
\end{proof}

\begin{lemma}\label{lemma:sumH_s'}
	We have for any $\bs, \bs' \in \Ff$, 
	\begin{align} 
	\sum_{\bk \in \pi_\bs} |H_{\bs'}(\by_{\bs;\bk})| \le  e^{K|\bs|_1},
	\label{sumH_s'}
	\end{align}
	where the constant $K$ is independent of $J$ and $\bs, \bs'$.
\end{lemma}

\begin{proof}
	From Cram\'er's  bound we deduce that (see, e.g., \cite [Lemma 3.2]{Dung21})	
	\begin{equation} \label{CramerBnd}
	|H_s(y)\sqrt{g(y)}| 
	\ < \ 1, \quad \forall y \in \RR, \ \forall s \in \NN_0,	
	\end{equation}
	or, equivalently,
		\begin{equation} \label{CramerBnd2}
	|H_s(y)|
	\ < \
	(2\pi)^{1/4} e^{y^2/4}, \quad \forall y \in \RR, \ \forall s \in \NN_0.	
	\end{equation}
	Let $\bs, \bs' \in \Ff$ and $\bk \in \pi_\bs$ be given. Notice that for the univariate Hermite polynomials, 
	$H_0 = 1$,  $H_{2s+1}(0) = 0$ and $|H_{2s}(0)|\le 1 $ for $s \in \NN_0$. Hence, we have by \eqref{CramerBnd2},
	\begin{align} 
	|H_{\bs'}(\by_{\bs;\bk})| 
	\le 
	\prod_{j \in \supp (\bs')\cap \supp (\bs)} |H_{s_j'}(y_{s_j,k_j})|
	\le 
	\prod_{j \in \supp (\bs)}  	(2\pi)^{1/4} e^{y_{s_j,k_j}^2/4}.
	\label{}
	\end{align}
	Therefore,
	\begin{align} 
	\sum_{\bk \in \pi_\bs} |H_{\bs'}(\by_{\bs;\bk})| 
	\le  
	\sum_{\bk \in \pi_\bs} \prod_{j \in \supp (\bs)} 	(2\pi)^{1/4} e^{y_{s_j,k_j}^2/4}
	= 
	\prod_{j \in \supp (\bs)} (2\pi)^{1/4} \sum_{k_j \in \pi_{s_j}} 	 e^{y_{s_j,k_j}^2/4}.
	\label{sum_{bk in pi_bs} |H_{bs'}|}
	\end{align}
	The inequalities \cite[(6.31.19)]{Sze39}  yield that
	\begin{equation} \label{}
	|y_{s;k}|
	\ \le \
	K_1 \frac{|k|}{\sqrt s}, \quad \forall k \in \pi_s, \ \forall s \in \NN.	
	\end{equation}
	Consequently,		
	\begin{equation} \label{}
	(2\pi)^{1/4} \sum_{k_j \in \pi_{s_j}} 	 e^{y_{s_j,k_j}^2/4}
	\ \le \		
	2 (2\pi)^{1/4} \sum_{k_j=0}^{\lfloor s_j/2\rfloor} 
	\exp{\left( \frac{K_1}{4} \frac{k_j^2}{s_j}\right)}
	\ \le \
	e^{Ks_j}, \quad \forall s_j \in \NN.	
	\end{equation}
	This allows us to finish the proof of the lemma  as	
	\begin{align} 
	\sum_{\bk \in \pi_\bs} |H_{\bs'}(\by_{\bs;\bk})| 
	\le  
	\prod_{j \in \supp (\bs)} 	 e^{K s_j}
	= e^{K |\bs|_1}.
	\notag
	\end{align}
	\hfill
\end{proof}

\begin{lemma}\label{lemma: normL_{s;k}<}
	We have for any $s \in \NN$  and $k \in \pi_{s}$, 
	\begin{align} 
	\norm{L_{s;k}}{L_2(\RR,\gamma)} 
	\le
	e^{K s},
	\label{norm-L_2}
	\end{align}
	and
	\begin{align} 
	\norm{L_{s;k}}{L_{\infty}^{\sqrt{g}}(\RR)} 
	\le
	e^{K s},
	\label{norm-L_infty}
	\end{align}
	where the constants $K$ are independent of $s$ and $k \in \pi_{s}$.
\end{lemma}

\begin{proof}
	Notice that $L_{s;k}$ is a polynomial having $s$ single zeros $\{y_{s;j}\}_{j \in \pi_{s}, \, j \not= k}$, and that $L_{s;k}(y_{s;k}) =1$. Moreover,
	there is no any zero in the open interval $(y_{s;k-1}, y_{s;k})$ and 
	$$
	L_{s;k}(y_{s;k}) = \max_{y \in [y_{s;k-1}, y_{s;k}]} L_{s;k}(y) = 1.
	$$
	Hence,
	\begin{align} 
	\label{|L_{s;k}(y)| < 1}
	|L_{s;k}(y)| \le 1, \quad \forall y \in [y_{s;k-1}, y_{s;k+1}].
	\end{align}
	
	Let us estimate $|L_{s;k}(y)|$ for $y \in \RR \setminus (y_{s;k-1}, y_{s;k+1})$.
	From the definition one can see that 
	\begin{align} 
	L_{s;k} (y)
	:=
	\prod_{k' \in \pi_s \ k'\not=k}\frac{y - y_{s;k'}}{y_{s;k} - y_{s;k'}}
	=
	A_{s;k} (y - y_{s;k})^{-1} H_{s+1}(y),
	\label{L_{s;k}(y)=}
	\end{align}
	where
	\begin{align} \label{A_{s;k}}
	A_{s;k}
	:=
	\big((s+1)!\big)^{1/2}\prod_{k' \in \pi_s \ k'\not=k}(y_{s;k} - y_{s;k'})^{-1}.
	\end{align}
	From the inequalities  \cite[(6.31.22)]{Sze39}
	\begin{align} \label{d_s}
	\frac{\pi \sqrt{2}}{ \sqrt{2s + 3}}
	\le d_s
	\le
	\frac{\sqrt{10.5}}{ \sqrt{2s + 3}}
	\end{align}
	for	the minimal distance $d_s$ between consecusive $y_{s;k}$, $k \in \pi_s$, we have that
	$$
	|y - y_{s;k}|^{-1} \le d_s^{-1} \le \frac{ \sqrt{2s + 3}}{\sqrt{10.5}} < \sqrt{s}, \quad \forall y \in \RR \setminus (y_{s;k-1}, y_{s;k+1}),
	$$
	and for any $s \in \NN$ and $k, k' \in \pi_{s}$ with $k' \not= k$,
	\begin{align} 
	|y_{s;k} - y_{s;k'}|^{-1}
	\le 
	C \frac{\sqrt{s}}{|k - k'|},
	\label{}
	\end{align}
	which yield for any $y \in \RR \setminus (y_{s;k-1}, y_{s;k+1})$,
	\begin{align} 
	|y - y_{s;k}|^{-1}|A_{s;k}|
	&\le
	\sqrt{s} \big((s+1)!\big)^{1/2} \prod_{k' \in \pi_s \ k'\not=k}|y_{s;k} - y_{s;k'}|^{-1}
	\le
	\sqrt{s} C^s \frac{\big((s+1)!\big)^{1/2}s^{s/2}}{k! (s-k)!}
	\notag
	\\
	&\le
	\sqrt{s} C^s \binom{s}{k}\frac{\big((s+1)!\big)^{1/2}s^{s/2}}{s! }
	\le
	\sqrt{s}(2C)^s \frac{\big((s+1)!\big)^{1/2}s^{s/2}}{s! }
	\le
	e^{K_1 s}.
	\label{|A_{s;k}|<}
	\end{align}
	In the last step we used the Stirling's approximation for factorial. Thus, we have proven that
	\begin{align} 
	|L_{s;k} (y)|
	\le 
	e^{K_1 s} |H_{s+1}(y)|, \quad \forall y \in \RR \setminus (y_{s;k-1}, y_{s;k+1}).
	\label{}
	\end{align}
	With $I_{s;k}:= [y_{s;k-1}, y_{s;k+1}]$, from the last estimate and \eqref{|L_{s;k}(y)| < 1} we prove \eqref{norm-L_2}:
	\begin{align} 
	\norm{L_{s;k}}{L_2(\RR,\gamma)}^2
	&=
	\norm{L_{s;k}}{L_2(I_{s;k},\gamma)}^2 + \norm{L_{s;k}}{L_2(\RR \setminus I_{s;k},\gamma)}^2
	\notag
	\\
	&
	\le
	1 + e^{2K_1 s}\norm{H_s}{L_2(\RR ,\gamma)}^2
	= 1 + e^{2K_1 s}
	\le 
	e^{2K s}.
	\notag
	\end{align}
	The inequality \eqref{norm-L_infty} can be proven similarly by using \eqref{CramerBnd}.
	\hfill
\end{proof}

\begin{lemma}\label{lemma: b_k<}
	Assume that p and q are polynomials on $\RR$ in the form 
	\begin{align} 
	p(y):= \sum_{k=0}^m a_k y^k, \quad 	q(y):= \sum_{k=0}^{m-1} b_k y^k,
	\label{}
	\end{align}
	and that $p(y) = (y - y_0) q(y)$ for a point $y_0 \in \RR$. Then we have
	\begin{align} 
	|b_k| 
	\le
	\sum_{k=0}^m |a_k|, \quad 	k = 0,..., m-1.
	\label{}
	\end{align}
\end{lemma}

\begin{proof}
	From the definition we have
	\begin{align} 
	\sum_{k=0}^m a_k y^k
	= 
	-b_0 y_0 + \sum_{k=0}^{m-1} (b_{k-1} - b_k y_0) y^k + b_{m-1} y^m.
	\label{}
	\end{align}
	Hence we obtain	
	\begin{align} 
	0 = a_0 + b_0 y_0; \quad b_k = a_{k+1} +  b_{k+1} y_0, \ k =1,..., m-2; \quad b_{m-1} = a_m.
	\label{a_k,b_k-equalities}
	\end{align}	
	From the last equalities one can see that the lemma is trivial if $y_0 = 0$. Consider the case  $y_0 \not= 0$.	
	If $|y_0| \le 1$, from \eqref{a_k,b_k-equalities} we deduce that 	
	\begin{align} 
	b_k = \sum_{j = k+1}^m a_j y_0^{j-k-1}.
	\label{}
	\end{align}
	and, consequently,
	\begin{align} 
	|b_k| \le  \sum_{j = k+1}^m |a_j| |y_0|^{j-k-1} \le \sum_{j = 0}^m |a_j|.
	\label{}
	\end{align}	
	If $|y_0| > 1$, from \eqref{a_k,b_k-equalities} we deduce that 	
	\begin{align} 
	b_k = - \sum_{j = 0}^k a_j y_0^{-(k+1-j)},
	\label{}
	\end{align}
	and, consequently,
	\begin{align} 
	|b_k| \le  \sum_{j = 0}^k |a_j| |y_0|^{-(k+1-j)} \le \sum_{j = 0}^m |a_j|.
	\label{}
	\end{align}	
	\hfill
\end{proof}

\begin{lemma}\label{lemma: sum |b_ell|}
	Let $b^{s;k}_\ell$ be the polynomial coefficients of $L_{s;k}$ as in the representation \eqref{L_s}.
	Then we have for any $s \in \NN_0$ and $k \in \pi_s$,
	\begin{align} 
	\sum_{\ell=0}^s |b^{s;k}_\ell|
	\le e^{Ks} s! \, ,	
	\end{align} 
	where the constant $K$ are independent of $s$ and $k \in \pi_{s}$.	 	
\end{lemma}

\begin{proof}
	For $s\in \NN_0$, we represent  the univariate Hermite polynomial $H_s$ in the form
	\begin{align}\label{H_s}
	H_s(y) :=	\sum_{\ell=0}^s a_{s,\ell} y^\ell.
	\end{align} 
	By using the well-known equality
	\begin{align}\label{H_s=}
	H_s(y)
	=
	s! \sum_{\ell=0}^{\left\lfloor \tfrac{s}{2} \right\rfloor} \frac{(-1)^\ell}{\ell!(s - 2\ell)!} \frac{y^{s -  2\ell}}{2^\ell},
	\end{align} 
	one can derive that 
	\begin{equation} \label{sum|a_s,ell|<}
	\sum_{\ell=0}^{s} |a_{s,\ell}| \leq s!.
	\end{equation}
	From \eqref{L_{s;k}(y)=} we have 
	\begin{align} 
	A_{s;k}H_{s+1}(y)
	= (y - y_{s;k})	L_{s;k}(y), 
	\label{	L_{s;k}=}
	\end{align}
	where $A_{s;k}$ is given as in \eqref{A_{s;k}}. By Lemma \ref{lemma: b_k<}, \eqref{sum|a_s,ell|<} and \eqref{|A_{s;k}|<}, we obtain
	\begin{align} 
	\sum_{\ell=0}^s |b^{s;k}_\ell|
	\le
	\sum_{\ell=0}^s A_{s;k}\sum_{\ell'=0}^{s+1} |a_{s+1,\ell'}|
	\le 
	e^{Ks} s!\,. 
	\end{align} 
	\hfill
\end{proof}

\begin{lemma}\label{l:Rm trun1}
	Let $\varphi (\by)= \prod_{j = 1}^m \varphi_j(y_j)$ 
	for  $\by \in \RR^m$, where $\varphi_j$ 
	is a polynomial in the variable $y_j$
	of degree not greater than $\omega$ for 
	$j=1,\ldots,m$. 
	Then there holds 
	\begin{align} 
	\norm{\varphi}{L_2(\RR^m{\setminus}B^m_\omega,\gamma)}
	\le 
	Cm
	\exp \left(- K\omega \right)
	\norm{\varphi}{L_2(\RR^m,\gamma)},
	\label{Rmsqu1}
	\end{align}
	and 	\begin{align} 
	\norm{\varphi}{L_\infty^{\sqrt{g}}(\RR^m{\setminus}B^m_\omega)}
	\le 
	Cm
	\exp \left(- K\omega \right)
	\norm{\varphi}{L_\infty^{\sqrt{g}}(\RR^m)},
	\label{Rmsqu2}
	\end{align}
	where the constants $C$ and $K$ are independent of $\omega$, $m$ and $\varphi$.
\end{lemma}

\begin{proof}
	The inequality \eqref{Rmsqu1} was proven in \cite[Lemma 3.3]{DNP21}. The inequality \eqref{Rmsqu2} can be proven in a similar way with a slight modification.
\hfill	
\end{proof}	

\subsection{Proof of Theorem \ref{lemma:coll-approx}}
\label{Proof of Theorem ref{lemma:coll-approx}}

\begin{proof} This theorem was proven in \cite[Corollary 3.11]{Dung21} for the case $U=\RRi$. Let us prove it for the case $U=\RRM$.
	By Lemma~\ref{lemma[AbsConv]}  the series \eqref{series} 
	converges  unconditionally  in  the  space $L_2(\RRM,X,\gamma)$ to $v$. 
	Observe that $I_{\Lambda(\xi)} H_\bs = H_\bs$ for every $\bs \in \Lambda(\xi)$ and
	$\Delta_\bs H_{\bs'} = 0$ for every $\bs \not\le \bs'$. Hence for the downward closed set $\Lambda(\xi) \subset \NMO$, we can write
	\begin{equation}\nonumber
	I_{\Lambda(\xi) }v
	\ = \
	I_{\Lambda(\xi) }\Big(\sum_{ \bs \in \NMO} v_\bs \,H_\bs \Big)
	\ =  \
	\sum_{ \bs \in \NMO} v_\bs \,I_{\Lambda(\xi) } H_\bs
	\ =  \
	S_{\Lambda(\xi)} v
	\ + \ \sum_{\bs \not\in \Lambda(\xi) } v_\bs \, I_{\Lambda(\xi)  \cap R_\bs}\, H_\bs,
	\end{equation}
	where $R_\bs:= \{\bs' \in \NMO: \bs' \le \bs\}$ and 
	$$
		S_{\Lambda(\xi)} v:= \sum_{\bs \in \Lambda(\xi) } v_\bs \, H_\bs
	$$
for $v \in L_2(\RRM,X,\gamma)$ represented by the Hermite gpc expansion \eqref{series}.
	This implies
	\begin{equation} \label{v-I}
	\big\|v- I_{\Lambda(\xi)} v\big\|_{L_{\infty}^{\sqrt{g}}(\RRM,X)}
	\ \le \
	\big\|v- S_{\Lambda(\xi)} v\big\|_{L_{\infty}^{\sqrt{g}}(\RRM,X)} 
	+  
	\sum_{\bs  \not\in \Lambda (\xi)}  
	\big\|I_{\Lambda(\xi) \cap R_\bs}\, H_\bs\big\|_{L_{\infty}^{\sqrt{g}}(\RRM)}.
	\end{equation}
	Therefore, to prove the lemma  it is sufficient to  show that each term in the right-hand side is bounded by 	$C\xi^{-(1/q - 1/2)}$.
	The bound of the first term can be obtained from the Cauchy--Schwasz inequality and \eqref{CramerBnd}:
	\begin{equation} \label{v-S}
	\begin{split}
	\|v- S_{\Lambda(\xi)}\|_{L_{\infty}^{\sqrt{g}}(\RRM,X)}
	&\le 
	\sum_{\sigma_{\bs}> \xi^{1/q} } \|v_\bs\|_{X} \, \|H_\bs\|_{L_{\infty}^{\sqrt{g}}(\RRM)}
	\le 
	\sum_{\sigma_{\bs}> \xi^{1/q} } \|v_\bs\|_{X} 
	\\[1.5ex]
	&\le 
	\left(\sum_{\sigma_{\bs}> \xi^{1/q} } (\sigma_{\bs}\|v_\bs\|_{X})^2\right)^{1/2}
	\left(\sum_{\sigma_{\bs}> \xi^{1/q} } \sigma_{\bs}^{-2}\right)^{1/2}
	\le 
	C\,
	\left(\sum_{\sigma_{\bs}> \xi^{1/q} }  \sigma_{\bs}^{-q} 
	\sigma_{\bs}^{-(2- q)}\right)^{1/2}
	\\[1.5ex]
	&\le 
	C \xi^{-(1/q - 1/2)}
	\left(\sum_{\bs \in \NN_0^M} \sigma_{\bs}^{-q} \right)^{1/2}
	\le C \xi^{-(1/q - 1/2)}.
	\end{split}
	\end{equation} 
	
	Let us prove the bound of the second term in the right-hand side of \eqref{v-I}. We have that
	\begin{equation} \label{norm-estimate}
	\big\|I_{\Lambda(\xi)  \cap R_\bs}\, H_\bs\big\|_{L_{\infty}^{\sqrt{g}}(\RRM)}
	\ \le \
	\sum_{\bs' \in \Lambda(\xi)  \cap R_\bs}\, \|\Delta_{\bs'} (H_\bs)\big\|_{L_{\infty}^{\sqrt{g}}(\RRM)}.
	\end{equation}
	We estimate the norms inside the  right-hand side.  For $\bs \in \NMO$	and $\bs' \in \Lambda(\xi)  \cap R_\bs$,  we have
	$
	\Delta_{\bs'} (H_\bs) = \prod_{j=1}^M \Delta_{s'_j} (H_{s_j}).
	$
	From Lemma  \ref{lemma[Delta_{bs}]} and \eqref{CramerBnd}  we deduce that
	\begin{equation} \nonumber
	\begin{split}
	\|\Delta_{s'_j} (H_{s_j})\|_{L_\infty^{\sqrt{g}}(\RR)}
	\ \le \
	(1 + C_\varepsilon s'_j)^{1/6 + \varepsilon} \, \|H_{s_j}\|_{L_\infty^{\sqrt{g}}(\RR)}
	\ \le \
	(1 + C_\varepsilon s'_j)^{1/6 + \varepsilon}, 
	\end{split}
	\end{equation}
	and consequently,
	\begin{equation}  \label{[|Delta_{bs'}|<]4}
	\|\Delta_{\bs'} (H_\bs)\big\|_{L_{\infty}^{\sqrt{g}}(\RRM)}
		\ = \
	\prod_{j=1}^M \|\Delta_{s'_j} (H_{s_j})\|_{L_\infty^{\sqrt{g}}(\RR)}
	\ \le \
	p_{\bs'}(\theta_1,\lambda)
	\ \le \ 
	p_\bs(\theta_1, \lambda),
	\end{equation}
	where $\theta_1 =  1/6 + \varepsilon$ and recall that  $\lambda = C_\varepsilon$. 
	Substituting $\|\Delta_{\bs'} (H_\bs)\big\|_{L_{\infty}^{\sqrt{g}}(\RRM)}$ in \eqref{norm-estimate} by the right-hand side of \eqref{[|Delta_{bs'}|<]4}  gives that
	\begin{align} \nonumber
	\big\|I_{\Lambda(\xi)  \cap R_\bs}\, H_\bs\big\|_{L_{\infty}^{\sqrt{g}}(\RRM)}
	\ &\le \
	\sum_{\bs' \in \Lambda(\xi)  \cap R_\bs}\, p_\bs(\theta_1, \lambda) 
	\ \le \
	|R_\bs|\, p_\bs(\theta_1, \lambda)
	\\
	&\le \
	p_\bs(1,1)\, p_\bs(\theta_1, \lambda)
	\ \le \
	p_\bs(\theta/2, \lambda).
	\notag
	\end{align}	
	By using of  the last estimates  and the assumption 
	$\norm{\bp(\theta/q,\lambda)\bsigma^{-1}}{\ell_q(\NN_0^M)} \le C < \infty$  with a positive constant $C$ independent of $M$,
	we derive 	the bound of the second term in the right-hand side of \eqref{v-I}:
	\begin{equation} \nonumber
	\begin{split}
	\sum_{\bs  \not\in \Lambda (\xi)} 
	\big\|I_{\Lambda(\xi) \cap R_\bs}\, H_\bs\big\|_{L_{\infty}^{\sqrt{g}}(\RRM)}
	&\ \le  \
	C \sum_{\bs  \not\in \Lambda (\xi)}  \|v_\bs\|_{X}  \, p_\bs(\theta/2,\lambda)
	\\
	\ &\le \
	C	\left(\sum_{\sigma_{\bs}> \xi^{1/q} } (\sigma_{\bs}\|v_\bs\|_{X})^2\right)^{1/2}
	\left(\sum_{\sigma_{\bs}> \xi^{1/q} } p_\bs(\theta/2,\lambda)^2 \sigma_{\bs}^{-2}\right)^{1/2}
	\\[1.5ex]
	\ &\le \
	C\,
	\left(\sum_{\sigma_{\bs}> \xi^{1/q} } p_\bs(\theta/2,\lambda)^2 \sigma_{\bs}^{-q} 
	\sigma_{\bs}^{-(2- q)}\right)^{1/2}
	\\[1.5ex]
	\ &\le \
	C \xi^{-(1/q - 1/2)}
	\left(\sum_{\bs \in \NN_0^M} p_\bs(\theta,\lambda) \sigma_{\bs}^{-q} \right)^{1/2} 
	\le C \xi^{-(1/q - 1/2)},
	\end{split}
	\end{equation} 
	which together with \eqref{v-I} and \eqref{v-S} proves the theorem.
	\hfill
\end{proof}

\bibliographystyle{abbrv}

\bibliography{AllBib.bib}
\end{document}